\documentclass[11pt]{article}
\usepackage{tikz}

\usepackage{authblk}
\usepackage{amsthm,amsfonts,amsmath,amstext,amssymb,amsopn}
	\newtheorem{lemma}{Lemma}
	\newtheorem{theorem}{Theorem}
	\newtheorem{corollary}{Corollary}
	
        \newtheorem{proposition}{Proposition}

        \newtheorem{remark}{Remark}
\usepackage[T3,T1]{fontenc}
    \DeclareSymbolFont{tipa}{T3}{cmr}{m}{n}
    \DeclareMathAccent{\tipa}{\mathalpha}{tipa}{16}
\usepackage{mathtools}
\usepackage[margin=1in]{geometry}
\usepackage{hyperref, cleveref}
\usepackage{graphicx}
\usepackage{subcaption}
\usepackage[inline]{enumitem}
\usepackage[mathscr]{euscript}
\usepackage[backend=biber, style=numeric, bibencoding=utf8, natbib=true, hyperref=true]{biblatex}
\usepackage{algorithm}
\usepackage{algpseudocode}
\usepackage{multirow}
\usepackage{blkarray}
\usepackage[toc,page]{appendix}
\usepackage{todonotes}
\usepackage{thmtools, thm-restate}
\usepackage{tikz, pgfplots}
    \definecolor{pyblue}{RGB}{31,119,180}
    \definecolor{pyred}{RGB}{199,30,45}
    \definecolor{pyorange}{RGB}{252,164,4}
    \pgfplotsset{compat=1.18}
\usepackage{xparse}

\usepackage{etoolbox}
	\let\summary\abstract
	
	\patchcmd{\summary}{\abstractname}{\summaryname}{}{}
	\patchcmd{\summary}{\abstractname}{\summaryname}{}{}
	\newcommand{\summaryname}{Summary}

\newif\ifnewauthor
\newauthortrue

\newcounter{affils}

\makeatletter
\def\chapterauthor[#1]#2{\ifnewauthor\else, \fi
       #2\textsuperscript{\hyperref[authaffil\the\value{section}.#1]{#1}}%
       \ifnewauthor\newauthorfalse\gdef\chapterauthors{#2}\else
               \g@addto@macro\chapterauthors{, #2}\fi
       \ignorespaces %
}
\makeatother

\def\chapteraffil[#1]{\item[#1 --]%
    \refstepcounter{affils}%
    \label{authaffil\the\value{section}.#1}%
}





\usepackage{stackengine}
	
	\newcommand{\preced}[1]{\Shortunderstack{{$\,\leftharpoonup\,$} {$#1$}}}
	\makeatletter %
	\newcommand{\Neginternal}[3]{\mathpalette\Neg@{{#1}{#2}{#3}}}
	\newcommand{\ssnegslash}[1]{\rotatebox[origin=c]{60}{$\m@th#1{\dabar@}\mkern-7mu{\dabar@}$}}
	\newcommand{\ssNeg}[2][0mu]{\Neginternal{#1}{\ssnegslash}{#2}}
	\newcommand*\Neg@[2]{\Neg@@{#1}#2}
	\newcommand*\Neg@@[4]{%
	  \mathrel{\ooalign{%
	    $\m@th#1#4$\cr
	    \hidewidth$\m@th#3{#1}\mkern\muexpr#2*2$\hidewidth\cr
	  }}%
	}
	\newcommand{\npreced}[1]{\Shortunderstack{{$\,\ssNeg[-1mu]{\leftharpoonup}\,$} {$#1$}}}
	\makeatother  %

\makeatletter
    \newcommand\shorteq{\mathrel{\mathpalette\shorteq@{.4}}}
    \newcommand{\shorteq@}[2]{%
      \resizebox{#2\width}{\height}{$\m@th#1=$}%
      }\makeatother

\usepackage{eqnarray,xstring}

\NewDocumentEnvironment{FeasRegion}{sO{center}}%
    {\IfBooleanTF{#1}
        {\IfEqCase{#2}{%
            {center}{\begin{equationarray*}{c@{\hskip2em}l}}
            {align}{\begin{equationarray*}{r@{\hskip0.5em}c@{\hskip0.5em}l@{\hskip2em}l}}
        }}
        {\IfEqCase{#2}{%
            {center}{\begin{equationarray}{c@{\hskip2em}l}}
            {align}{\begin{equationarray}{r@{\hskip0.5em}c@{\hskip0.5em}l@{\hskip2em}l}}
    }}}
    {\IfBooleanTF{#1}
        {\end{equationarray*}\ignorespacesafterend}
        {\end{equationarray}\ignorespacesafterend}
    }

\NewDocumentEnvironment{LPArray}{sO{center}mm} %
    {\IfBooleanTF{#1}
        {\IfEqCase{#2}{
            {center}{\begin{equationarray*}{rc@{\hskip2em}l}
                \text{#3} & \multicolumn{2}{l}{#4} \\
         	\text{s.t.}}
            {align}{\begin{equationarray*}{rr@{\hskip0.5em}c@{\hskip0.5em}l@{\hskip2em}l}
                \text{#3} & \multicolumn{4}{l}{#4} \\
         	\text{s.t.}}
        }}
        {\IfEqCase{#2}{
            {center}{\begin{equationarray}{rc@{\hskip2em}l}
                \text{#3} & \multicolumn{2}{l}{#4} \\
         	\text{s.t.}}
            {align}{\begin{equationarray}{rr@{\hskip0.5em}c@{\hskip0.5em}l@{\hskip2em}l}
                \text{#3} & \multicolumn{4}{l}{#4} \\
         	\text{s.t.}}
    }}}
    {\IfBooleanTF{#1}
        {\end{equationarray*}\ignorespacesafterend}
        {\end{equationarray}\ignorespacesafterend}
    }

\makeatletter
\newcounter{Para}
\newcounter{NamedPara}
\NewDocumentCommand\ParaTitle{mg}%
    {\hfill\break\indent\textbf{#1}
        \IfValueTF{#2}
            {\textbf{(#2)}\refstepcounter{NamedPara}\def\@currentlabel{#2}}
            {\refstepcounter{Para}}
    }

\newcounter{NamedEq}
\NewDocumentEnvironment{NamedSubs}{mO{\alph}}
    {\begin{subequations}\refstepcounter{NamedEq}\def\@currentlabel{#1}\renewcommand{\theequation}{#1.#2{equation}}}
    {\end{subequations}\ignorespacesafterend\unskip\addtocounter{equation}{-1}}

\makeatother

\newcommand{\tvect}[2]{%
  \ensuremath{\big[\negthinspace\begin{smallmatrix}#1\\#2\end{smallmatrix}\big]}}

\newcommand{\PM}{\textup{P}}%
\newcommand{\LB}{\textup{LB}} %
\newcommand{\UB}{\textup{UB}} %
\newcommand{\BM}{\textup{M}} %

\def\hquad{\hskip0.5em\relax}
\def\quad{\hskip1em\relax}
\def\qquad{\hskip2em\relax}

\newcommand{\LS}{\hquad<\hquad}
\newcommand{\GS}{\hquad>\hquad}
\newcommand{\LEQ}{\hquad\leq\hquad}
\newcommand{\GEQ}{\hquad\geq\hquad}
\newcommand{\EQ}{\hquad=\hquad}
\newcommand{\IN}{\hquad\in\hquad}

\newcommand{\disj}{\mathcal{D}}
\newcommand{\numobjs}{N}
\newcommand{\setobjs}{[\![\numobjs]\!]}
\newcommand{\object}{\mathscr{B}}

\newcommand{\pairs}{\mathscr{P}}%

\newcommand{\combs}{\mathscr{F}}

\newcommand{\inst}[2]{\left(\textup{\ref*{#1}}^{#2}\right)}

\newcommand{\boolcomp}{\hspace{-1pt}\shorteq\hspace{.2pt}!\hspace{-2.8pt}\shorteq\hspace{-1pt}}

\newcommand{\bcfSymb}{h}
\NewDocumentCommand\bcf{e{_^}d<>g}{%
        \IfValueT{#3}{#3}%
    \bcfSymb
        \IfValueT{#1}{_{#1}}%
        \IfValueT{#2}{^{#2}}%
    \hspace{-1pt}
    \IfValueT{#4}{\left(#4\right)}%
}

\newcommand{\m}{\sigma}
\renewcommand{\d}{w}
\renewcommand{\r}{W}

\usepackage{pifont}
\newcommand{\cmark}{\ding{51}} %
\newcommand{\xmark}{\ding{55}} %

\bibliography{Bibliography.bib}
\graphicspath{{Figures}}

\title{Automating Idealness Proofs for Binary Programs with Application to Rectangle Packing\thanks{J. Fravel and R. Hildebrand were partially supported by AFOSR grant FA9550-21-1-0107. Any opinions, findings, and conclusions or recommendations expressed in this material are those of the authors and do not necessarily reflect the views of the Air Force Office of Scientific Research.}
}
\author[1]{Jamie Fravel\thanks{Correspondence to: jfravel@vt.edu}}
\author[1]{Robert Hildebrand}
\affil[1]{Grado Department of Industrial and Systems Engineering, Virginia Tech, Blacksburg, Virginia, USA}
\date{}

\begin{document}
\maketitle
\begin{abstract}
We develop an optimization framework for identifying ideal Mixed Binary Linear Programs (MBLP) which is linear when using known input data and nonconvex quadratic over parametric input data. These techniques are applied to various formulations for rectangle packing, conjectured to be pairwise-ideal. Additionally, we address a variation of the rectangle packing problem which incorporates clearances along selected edges of the packed objects. We present both existing and novel MBLP formulations for the underlying disjunctive program and investigate the poor performance of Gurobi's default branch-and-cut methodology. We operate under a strip-packing objective that aims to minimize the overall height of the packed objects.
\end{abstract}
Proof, Ideal Formulation.\\

\noindent \textbf{Data Disclosure:} The data that support the findings of this study are openly available in GitHub at \hyperlink{https://github.com/jfravel/Ideal-O-Matic}{https://github.com/jfravel/Ideal-O-Matic}.  Only synthetic data, generated specifically for this study, was used.

\section{Introduction}
In the field of operations research and optimization, integer programming stands as a pivotal method for solving complex decision-making problems where variables are constrained to be integers. One of the key challenges in integer programming is the development of ideal formulations, which are mathematical models that accurately represent the problem while being computationally efficient. Ideal formulations are crucial as they ensure that the feasible region of the problem is well-defined, leading to optimal solutions that are both precise and reliable. Moreover, such formulations often enhance the performance of optimization algorithms by reducing computational time and improving convergence rates. 
To address these challenges, we develop new computer-aided techniques to prove that mixed binary linear programs are ideal. 
We apply these techniques to various optimization models for rectangle packing problems.

The rectangle packing problem involves arranging a set of rectangles within a specified region such that no two rectangles overlap. This problem is a fundamental issue in various fields such as logistics, manufacturing, and electronics, where efficient space utilization is critical~\cite{Faina2020, LODI2002241}. The primary objective is to maximize the usage of the available area or minimize the required area to fit all the rectangles, leading to applications in cutting stock problems, loading problems, and VLSI design.

In addition to the basic non-overlapping constraint, rectangle packing problems often incorporate additional requirements such as clearances. A \textit{clearance} refers to the designated space around a rectangle where no other rectangle can be placed. These clearances are crucial in practical contexts such as allowing maintenance access to machinery in a factory or preventing interference between components on a printed circuit board. The clearances of two or more objects are traditionally allowed to overlap as long as none of the rectangles overlap the clearance of another, ensuring both safety and functionality in the packed layout.

\subsection{Literature}\label{sec:literature}

A significant body of literature has addressed \textit{row-based} layout problems with clearances, which are distinct from our setting. For instance, Keller et al.~\cite{ConstructionHeuristicsClearances2019Keller} and Uruk and Akbilek~\cite{ZoningMutualClearances2022UrukAkbilek} explore construction heuristics and zoning techniques for mutual clearances, respectively. Wan et al.~\cite{HybridGRASPMultiRow2022WanZuoLiZhao} present a hybrid GRASP for multi-row layouts, while Yu et al.~\cite{TabuSearchSingleRow2014YuZuoMurray} utilize tabu search for single-row problems. Additionally, Zuo et al.~\cite{SharingClearancesLayout2016ZuoMurraySmith} investigate clearance sharing in layout designs.

Guan et al.~\cite{ClearanceBoundsLayout2019GuanZhangLiu} examine a multi-floor layout problem with clearances by developing a Mixed-Integer Programming (MIP) model. Safarzadeh and Koosha~\cite{FuzzyClearnaces2017SafarzadehKoosha} introduce fuzzy clearances and propose a MIP alongside a genetic algorithm.

McKendall and Hakobyan~\cite{ApplicationFixedLayout2021McKendallHakobyan} employ a genetic algorithm to quickly generate good solutions for the Rectangle Packing Problem (RPP), though it shows inconsistency in convergence compared to an MIP-based approach. Ingole and Singh~\cite{FireflyFixedLayout2017IngoleSingh} propose a "Firefly" algorithm, later generalized into a "Biogeography-based" metaheuristic~\cite{BiogeographyFlexibleLayout2021IngoleSingh}, both of which yield promising results.

In some scenarios, objects are assumed to be fixed to a grid or "dotted board," enabling the use of covering models~\cite{RodriguesToledo2017CliqueCovering} or discrete genetic algorithms for irregularly shaped objects~\cite{PairwiseClusteringIrregularStrip2018SatoBauab, BiasedRandomKeyGeneticAlgorithmDotted2020JuniorCosta}.

Nonlinear programs utilizing IPOPT allow for free rotations of convex and non-convex polygons~\cite{IrregularFreeStripPacking2018PeraltaAndrettaOliveira}. These approaches are compared with other methods that also permit free rotations~\cite{NestingIrregularCuttingStock2016LiaoEtAl, CuttingPackingIrregular2016StoyanPankratovRomanova}. They employ a heuristic known as \emph{no-fit raster}~\cite{GeneticRasterNesting2017MundimAndrettaAlves} to generate initial solutions, subsequently refined using IPOPT.

Our work most closely follows the work of Huchette, Dey, and Vielma~\cite{StrongFloorLayout2017HuchetteVielma} on strong integer programming formulations for the floor layout problem.  They present several models to handle packing type constraints and study when some of their formulations are ideal but leave some open questions to be answered.

\subsection{Disjunctive model for Rectangle Packing}
We formally define the motivating problems that we're working with. Let $\setobjs = [1, \dots, N]$.

\ParaTitle{Rectangular Packing Problem}{RPP}\label{prob:RPP} \\
\textbf{Input:} A rectangular region with dimensions $\r_x\times \r_y$ and a collection $\object=\{\object_i\}_{i=1}^N$ of rectangular objects  with dimensions $\d_{ix}\times \d_{iy}$ for each $i\in\setobjs$. 
\\
\textbf{Output:} A placement of the center $[c_{ix},c_{iy}]$ of each object $\object_i\in\object$ such that all of the objects fit in the region and no object overlaps another. Alternatively, an assertion that no such placement exists.\\
~\\
\ref{prob:RPP} is modeled as a disjunctive program below:
    \begin{subequations}\label{model:basic}
    \begin{FeasRegion}
    \tfrac{1}{2}\d_{is}  \LEQ  c_{is}  \LEQ  \r_s - \d_{is}
        &\forall\ i\in\setobjs,\ s\in\{x,y\} \label{basic:otf} \\
    \bigvee_{(k,l,s)\in\combs_{ij}} (c_{ks} + \tfrac{1}{2}\d_{ks}   \LEQ  c_{ls} - \tfrac{1}{2}\d_{ls})
        &\forall\ (i,j)\in \pairs \label{basic:disj}
    \end{FeasRegion}
    \end{subequations}
    
\noindent where $\pairs = \{(i,j)\in\setobjs\times\setobjs:i<j\}$ is the set of all distinct pairs of objects, and $\combs_{ij} = \{(i,j,x),(j,i,x),(i,j,y),(j,i,y)\}$ is the set of combinations of $i$, $j$, and the directions $\{x,y\}$. This is a disjunctive programming problem that can be formulated (or embedded) as a Mixed-Binary Linear Program (MBLP) in a variety of ways which will be discussed in Section \ref{sec:embeddings}. Now we introduce and justify an additional feature of our variant of rectangle packing.

\subsection{Clearances}\label{sec:margins}
We include a \textit{clearance} on each face of each object, given by a vector $\mathbf{\m}_i = [\m_{ix}^-,\m_{iy}^-,\m_{ix}^+,\m_{iy}^+]^\top\in\mathbb{R}^4$ for each $i\in\setobjs$, to represent some free space required required to properly operate or maintain whatever is represented by said object. The clearances should lie entirely within the region and are allowed to overlap so long as no object occludes the clearance of another; see Figure \ref{fig:Occlusion} for an explanation of occlusion and Figure \ref{fig:MarginDimension} for a visualization of the clearance parameters.
    \begin{figure}[!ht]\centering
        \begin{subfigure}{0.4\textwidth}\centering
        \includegraphics[width=0.5\textwidth]{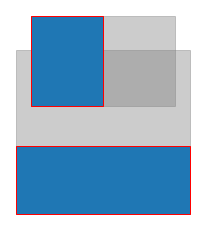}
        \caption{An invalid solution due to object/clearance overlap}
        \label{fig:InvalidOcclusion}
        \end{subfigure}
    ~
        \begin{subfigure}{0.4\textwidth}\centering
        \includegraphics[width=0.5\textwidth]{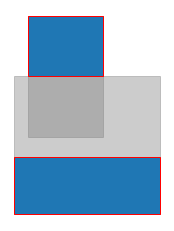}
        \caption{A valid solution}
        \label{fig:ValidOcclusion}
        \end{subfigure}
    \caption{Examples of Clearance Occlusion}
    \label{fig:Occlusion}
    \end{figure}

Clearance is an important feature of many practical layout/packing problems: manufacturing machine layouts may require clearances to reduce vibrations in neighboring machines, adequately ventilate the workspace, provide access to maintenance panels, or simply to allow for manual operation; circuit boards may require clearances to improve heat dissipation or mitigate crosstalk interference. The disjunctive programming model for \ref{prob:RPP} with clearances is given by

    \begin{subequations}\label{model:mbasic}
    \begin{FeasRegion}
    \tfrac{1}{2}\d_{is} + \m_{is}^-  \LEQ  c_{is}  \LEQ  \r_s - \d_{is} - \m_{is}^+
        &\forall\ i\in\setobjs,\ s\in\{x,y\} \label{mbasic:otf} \\
    \bigvee_{(k,l,s)\in\combs_{ij}} (c_{ks} + \tfrac{1}{2}\d_{ks}   \LEQ  c_{ls} - \tfrac{1}{2}\d_{ls} - \max\{\m_{ks}^+,\m_{ls}^-\})  
        &\forall\ (i,j)\in \pairs \label{mbasic:disj}.
    \end{FeasRegion}
    \end{subequations}

\noindent Notice that if $\mathbf{\m}_i = \boldsymbol{0}$ for each $i\in\setobjs$, then \eqref{model:mbasic} reduces exactly to \eqref{model:basic}. In the pursuit of notational simplicity, we define the \textit{Precedence Margin} $\PM_{ijs} = \frac{1}{2}\d_{is} + \frac{1}{2}\d_{js} + \max\{\m_{is}^+,\m_{js}^-\}$ for each $(i,j,s)\in\pairs\times\{x,y\}$; see Figure \ref{fig:Parameters} for a visualization of this parameter. Additionally,  define the lower bound $\LB_{is} = \frac{1}{2}\d_{is} + \m_{is}^-$ and upper bound $\UB_{is} = \r_s - \frac{1}{2}\d_{is} - \m_{is}^+$ for each object $\object_i$ in both directions $s\in\{x,y\}$. We use capital letters for these parameters to emphasize their composite nature. 

    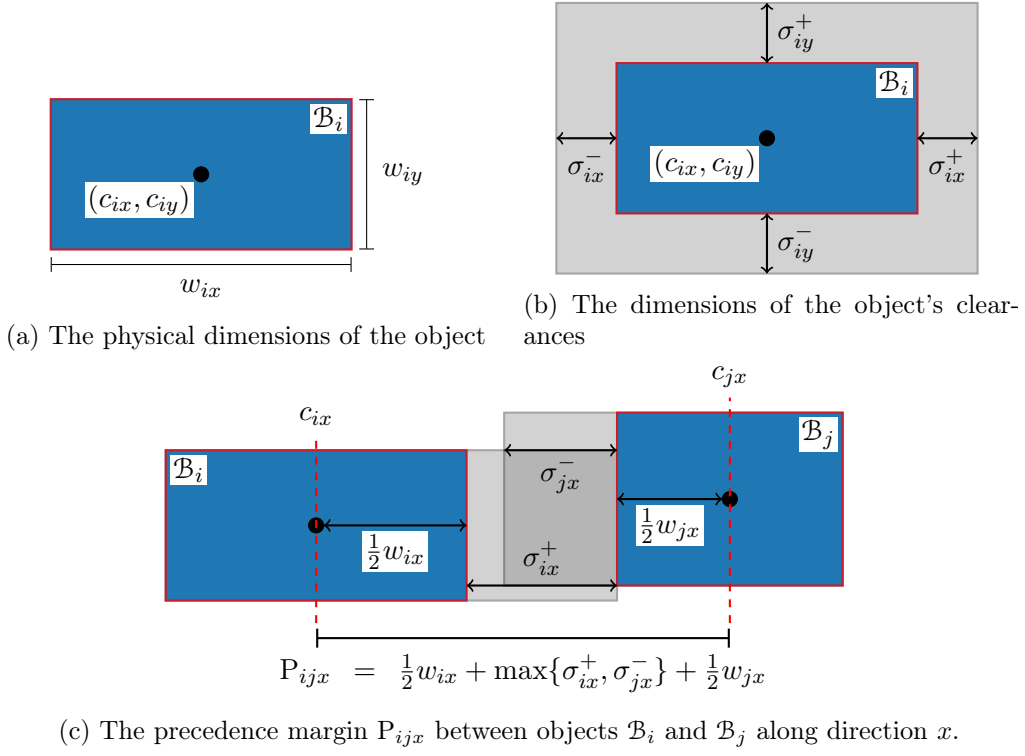
\begin{figure}[!ht]\centering
        \begin{subfigure}{0.4\textwidth}\centering
        \begin{tikzpicture}
\filldraw[fill=pyblue, draw=pyred, thick] (0,0.5) rectangle (4,2.5) node[fill=white, below left=2pt, inner sep=1pt] {$\object_i$};
\draw[fill=black] (2,1.5) circle (0.1) node[fill=white, below left=4pt, inner sep=1pt] {$(c_{ix}, c_{iy})$};
\draw[|-|] (0,0.3) -- (4,0.3) node[midway, below=2pt] {$\d_{ix}$};
\draw[|-|] (4.2,0.5) -- (4.2,2.5) node[midway, right=2pt] {$\d_{iy}$};
\end{tikzpicture}
        \caption{The physical dimensions of the object}
        \label{fig:PhysicalDimension}
        \end{subfigure}
    ~
        \begin{subfigure}{0.4\textwidth}\centering
        \begin{tikzpicture}
\filldraw[fill=gray, draw=black, thick, opacity=.35] (0,0) rectangle (5.6,3.6);
\filldraw[fill=pyblue, draw=pyred, thick] (0.8,0.8) rectangle (4.8,2.8) node[fill=white, below left=2pt, inner sep=1pt] {$\object_i$};
\draw[fill=black] (2.8,1.8) circle (0.1) node[fill=white, below left=4pt, inner sep=1pt] {$(c_{ix}, c_{iy})$};
\draw[<->, thick] (2.8,2.8) -- (2.8,3.6) node[midway, right] {$\m_{iy}^+$};
\draw[<->, thick] (2.8,0) -- (2.8,0.8) node[midway, right] {$\m_{iy}^-$};
\draw[<->, thick] (4.8,1.8) -- (5.6,1.8) node[midway, below] {$\m_{ix}^+$};
\draw[<->, thick] (0,1.8) -- (0.8,1.8) node[midway, below] {$\m_{ix}^-$};
\end{tikzpicture}
        \caption{The dimensions of the object's clearances}
        \label{fig:MarginDimension}
        \end{subfigure}
    \\
        \begin{subfigure}{1\textwidth}\centering
        \begin{tikzpicture}
\filldraw[fill=gray, draw=black, thick, opacity=.35] (0,0.5) rectangle (6,2.5);
\filldraw[fill=pyblue, draw=pyred, thick] (0,2.5) node[fill=white, below right=2pt, inner sep=1pt] {$\object_i$} rectangle (4,0.5) ;
\draw[fill=black] (2,1.5) circle (0.1);
\filldraw[fill=gray, draw=black, thick, opacity=.35] (4.5,0.7) rectangle (9,3);
\filldraw[fill=pyblue, draw=pyred, thick] (6,0.7) rectangle (9,3) node[fill=white, below left=2pt, inner sep=1pt] {$\object_j$};
\draw[fill=black] (7.5,1.85) circle (0.1);
\draw[<->, thick] (6,1.85) -- (7.4,1.85) node[midway, below=2pt, fill=white, inner sep=1pt] {$\tfrac{1}{2}\d_{jx}$};
\draw[<->, thick] (2.1,1.5) -- (4,1.5) node[midway, below=2pt, fill=white, inner sep=1pt] {$\tfrac{1}{2}\d_{ix}$};
\draw[<->, thick] (4,0.7) -- (6,0.7) node[midway, above] {$\m_{ix}^+$};
\draw[<->, thick] (4.5,2.5) -- (6,2.5) node[midway, below=-2pt] {$\m_{jx}^-$};
\draw[dashed, thick, red] (2,0.2) -- (2,2.7) node[above, black] {$c_{ix}$};
\draw[dashed, thick, red] (7.5,0.2) -- (7.5,3.2) node[above, black] {$c_{jx}$};
\draw[|-|, thick] (2,0) -- (7.5,0) node[midway, below=2pt] {$\PM_{ijx} \EQ \tfrac{1}{2}\d_{ix} + \max\{\m_{ix}^+,\m_{jx}^-\}+\tfrac{1}{2}\d_{jx}$};
\end{tikzpicture}
        \caption{The precedence margin $\PM_{ijx}$ between objects $\object_i$ and $\object_j$ along direction $x$.}
        \label{fig:PrecedenceMargin}
        \end{subfigure}
    \caption{Visualization of the parameters}
    \label{fig:Parameters}
    \end{figure}
Given this notation, Model \ref{model:mbasic} collapses to
    \begin{subequations}\label{model:pbasic}
    \begin{FeasRegion}
    c_{is}  \IN  [\LB_{is},\UB_{js}]
        &\forall\ i\in\setobjs,\ s\in\{x,y\} \label{pbasic:otf} \\
    \bigvee_{(k,l,s)\in\combs_{ij}} (c_{ks} + \PM_{kls}  \LEQ  c_{ls})  
        &\forall\ (i,j)\in \pairs \label{pbasic:disj}.
    \end{FeasRegion}
    \end{subequations}

\noindent Conveniently, additional constraints on displacement between pairs of rectangles could be encoded by choosing precedence margins larger than they are defined here.
Our theoretical work applies to this more general problem even while our computational experiments all define precedence margins defined based on clearances.

\subsection{Objective Function}\label{sec:objectives}
Our study of the constraints is directly applicable to the MBLP embeddings of our disjunctive constraints, independent of any specific objective. In our computational experiments, we focus on the strip-packing objective as defined below.

\ParaTitle{The Rectangular Strip-Packing Problem}{SPP}\label{prob:SPP}\\
A version of \ref{prob:RPP} which seeks the minimum value of $\r_y$ under which a packing of the objects $\object$ into the $\r_x\times \r_y$ region is feasible. This is implemented as follows:
    \begin{subequations}
    \begin{LPArray}{Minimize}{h}
    & h  \GEQ  c_{ix} + \tfrac{1}{2}\d_{ix} + \m_{ix}^+
        &\forall\ i\in\setobjs \\
    &    c_{is}  \IN  [\LB_{is},\UB_{js}]
        &\forall\ i\in\setobjs,\ s\in\{x,y\} \label{spp:otf} \\
    &\bigvee_{(k,l,s)\in\combs_{ij}} (c_{ks} + \PM_{kls}  \LEQ  c_{ls})  
        &\forall\ (i,j)\in \pairs \label{spp:disj}.
    \end{LPArray}
    \end{subequations}

\noindent We chose this objective for its simplicity; it requires no additional disjunctive constraints and only one additional continuous variable $h$. This allows us to compare the different MBLP embeddings of \eqref{spp:disj} relatively directly. We maintain a hard upper bound $\r_y$ based on a greedy solution described in Section \ref{sec:CutsTest}.

\subsection{Contributions}
Our theoretical results are largely focused on the idea of \textit{pairwise-idealness} which we define here. A mixed-integer program is called \textit{ideal} if its continuous relaxation has no fractional extreme points. That is, if the program has $m$ continuous variables and $n$ integer variables, then assuming tightness on a linearly independent collection of $m+n$ of the relaxation's constraints necessarily implies integer values for each of the $n$ integer variables. We say that a mixed-integer programming formulation for \ref{prob:RPP} is \textit{pairwise-ideal} if it is ideal when restricted to only two objects.

We consider a number of distinct mixed-binary linear programming (MBLP) formulations (embeddings) for pairwise non-overlapping constraints which arise in \ref{prob:SPP} and a number of other contexts. We certify pairwise-idealness of some formulations and provide counterexamples for the remainder. In pursuit of these results, we develop a novel computational approach to demonstrating that MBLPs are ideal under a range of parameters which provides computer-aided proofs of our results. Specifically, we show the following:
    \begin{itemize}
    \item A computational proof that the commonly used "Standard Unary" formulation is pairwise-ideal. There already exists an analytic proof of this result for the Floor Layout Problem \cite{StrongFloorLayout2017HuchetteVielma}; an adaptation of this proof to \ref{prob:RPP} is given in Appendix \ref{sec:analyticSU}.
    \item A computational proof of pairwise-idealness for the "Refined Unary" formulation given by~\cite{StrongFloorLayout2017HuchetteVielma}; confirming their conjecture.
    \item A counterexample of pairwise-idealness for the binary formulation originally given by Meller, Chen, and Sherali ~\cite{SequencePair2007MellerChenSherali}.
    \item A novel modification of the binary formulation and a computational proof that the modified version is pairwise-ideal.
    \end{itemize}
We perform several numerical experiments on these formulations to determine the importance of implementing a pairwise-ideal formulation in settings with large numbers of objects. We also vary the models by adding known symmetry-breaking inequalities, experimenting with the impact of Gurobi's Cuts and Heuristics parameters, and incorporating branching priorities to combat the poor performance of Gurobi's branch-and-cut algorithm on these problems.

\subsection{ Outline}\label{sec:outline}
Section~\ref{sec:embeddings} introduces the four MIP formulations under consideration, followed by a discussion on data generation and the initial experiment. Section~\ref{sec:ModelVariations} explores various modifications to the formulations, demonstrating their efficacy through a second experiment. Section~\ref{sec:IdealOMatic} details the construction of our computer-aided framework for proving idealness, which is subsequently applied to the formulations in Section~\ref{sec:compproofs}. Section~\ref{sec:Conclusions} summarizes our findings and outlines future research directions. Finally, the appendix includes additional result tables and plots from our computational experiments, an analytical proof of pairwise-idealness for the standard unary formulation, and enhanced versions of some minor results.

\section{MBLP Embeddings of the Disjunctive Program}\label{sec:embeddings}
Following notation from \cite{StrongFloorLayout2017HuchetteVielma}, we say that $\object_i$ precedes $\object_j$ along direction $s$ (denoted $i\preced{s} j$) if $c_{is} + \PM_{ijs} \leq c_{js}$; that is, if $\object_i$ lies far enough ahead of $\object_j$ along direction $s$ so that neither overlaps or occludes the other. Constraint \ref{pbasic:disj} is enforcing the following, four-term disjunction
    $$
    \disj_{ij}^4 \quad=\quad (i \preced{x} j) 
        \hquad\vee\hquad (i \preced{y} j) 
        \hquad\vee\hquad (j \preced{x} i) 
        \hquad\vee\hquad (j \preced{y} i)
    $$
for each pair of objects $(i,j)\in\pairs$. Any mixed-binary linear description of $\disj_{ij}^4$ will be sufficient to write a complete MBLP formulation for \ref{prob:SPP} or any other rectangle packing problem with clearances. This can be done via an extension to a space with many continuous auxiliary variables (see \citet{DisjunctiveProgrammingHull1983Balas}) or via an embedding onto a space of binary variables (see \citet{MILPTechniques2015Vielma}).

\subsection{Unary Embeddings}
The first formulation we present is widely known, but we adapt the name given to it by \cite[Eq. (17)]{StrongFloorLayout2017HuchetteVielma}.  Our problem variation considers the dimensions of the rectangles as parameters rather than variables.  Hence, $\d$ is fixed in our models.

    \ParaTitle{The Standard Unary Formulation}{SU}\label{model:SU}
    \begin{NamedSubs}{SU}\begin{FeasRegion}
    c_{ls}  \GEQ  \LB_{ls} + (\LB_{ks} + \PM_{kls} - \LB_{ls})\delta_{kls}	 
        & \forall\ (k,l,s)\in\combs_{ij} \label{SU:lb}   \\
    c_{ks}  \LEQ  \UB_{ks} + (\UB_{ls} - \PM_{kls} - \UB_{ks})\delta_{kls}	 
        & \forall\ (k,l,s)\in\combs_{ij} \label{SU:ub}   \\
    c_{ks} - c_{ls}  \LEQ  \UB_{ks} - \LB_{ls} + (\LB_{ls} - \PM_{kls} - \UB_{ks})\delta_{kls}   
        & \forall\ (k,l,s)\in\combs_{ij} \label{SU:prec} \\
    \delta_{ijx} + \delta_{jix} + \delta_{ijy} + \delta_{jiy}  \EQ  1 	\label{SU:disj} \\
    \delta_{kls}  \IN  \{0,1\}   
        & \forall\ (k,l,s)\in\combs_{ij} \label{SU:indic}
    \end{FeasRegion}\end{NamedSubs}
where the indicator variable $\delta_{kls}$ takes value $1$ if $k\preced{s}l$ and $0$ otherwise. Recall that $\LB_{ks} = \frac{1}{2}\d_{ks} + \m_{ks}^-$, $\UB_{ks} = \r_s - \frac{1}{2}\d_{ks} - \m_{ks}^+$, and $\PM_{kls} = \frac{1}{2}\d_{ks} + \frac{1}{2}\d_l + \max\{\m_{ks}^+,\m_{ls}^-\}$ are problem parameters.

\begin{restatable}{theorem}{SUideal}\label{thm:SUideal}
\ref{model:SU} is pairwise-ideal if $\PM_{kls} < \UB_{ls} - \LB_{ks}$ for each $(k,l,s)\in\combs_{ij}$.
\end{restatable}

An analytic proof for Theorem \ref{thm:SUideal} is available in Appendix \ref{sec:analyticSU}. This proof is a lengthy case analysis that is largely adapted from a proof in \cite[Appendix B]{StrongFloorLayout2017HuchetteVielma} for a similar formulation to the Floor Layout Problem. We present a computational proof in Section \ref{sec:SUideal}.

Also described in \cite[Eq. (19)]{StrongFloorLayout2017HuchetteVielma} is an embedding called the Refined Unary formulation which we present adapted to our setting and notation:
    \ParaTitle{The Refined Unary Formulation}{RU}\label{model:RU}
    \begin{NamedSubs}{RU}\begin{FeasRegion}
    c_{ls}  \GEQ  \LB_{ls} + (\LB_{ks} + \PM_{kls} - \LB_{ls})\delta_{kls}	
        & \forall\ (k,l,s)\in\combs_{ij} 	\label{RU:lb}   \\
    c_{ks}  \LEQ  \UB_{ks} + (\UB_{ls} - \PM_{kls} - \UB_{ks})\delta_{kls}	
        & \forall\ (k,l,s)\in\combs_{ij}  \label{RU:ub}   \vspace{.5em}\\
    \begin{aligned}
    c_{ks} - c_{ls}  \LEQ  \PM_{lks} &- (\PM_{lks} + \PM_{kls})\delta_{kls} \\
        &\quad + (\UB_{ks} - \PM_{lks} - \LB_{ls})\delta_{lks}
    \end{aligned}	
        & \forall\ (k,l,s)\in\combs_{ij} 	\label{RU:prec} \vspace{.5em} \\
    \delta_{ijs} + \delta_{jis}  \LEQ  1 & \forall\ s \in \{x,y\}	\label{RU:tight} \\
    \delta_{ijx} + \delta_{jix} + \delta_{ijy} + \delta_{jiy}  \GEQ  1 \label{RU:disj} \\
    \delta_{kls}  \IN  \{0,1\} & \forall\ (k,l,s)\in\combs_{ij}  \label{RU:indic}.
    \end{FeasRegion}\end{NamedSubs}
Constraint \eqref{RU:prec} contains two indicator variables and takes on three distinct states:
    \begin{equationarray*}{c@{\hskip20pt}c@{\hskip20pt}c@{\hskip30pt}c@{\hskip30pt}c}
    \delta_{kls} = 1,\ \delta_{lks} = 0  & \Rightarrow & c_{ls}  \GEQ  c_{ks} + \PM_{kls} & \Rightarrow & (k\preced{s}l)\\
    \delta_{kls} = 0,\ \delta_{lks} = 1  & \Rightarrow &  c_{ks} - c_{ls}  \LEQ  \UB_{ks} - \LB_{ls} & \Rightarrow & \text{inactive}\\
    \delta_{kls} = 0,\ \delta_{lks} = 0  & \Rightarrow &  c_{ks}  \LEQ  c_{ls} + \PM_{lks} & \Rightarrow & (l\npreced{s}k).
    \end{equationarray*}
The third state, which does not exist in \ref{model:SU}, adds strong spatial discretization to the model without requiring any additional variables. Notice that four binary variables are enough to encode $2^4 = 16$ distinct states. This formulation is an embedding of \ref{prob:RPP} using an eight-term disjunction $\disj_{ij}^8$ and has been shown to work very well in practice. The following theorem was conjectured in~\cite{StrongFloorLayout2017HuchetteVielma}, and we resolve their conjecture affirmatively here.

\begin{restatable}{theorem}{RUideal}\label{thm:RUideal}
\ref{model:RU} is pairwise-ideal if $\PM_{kls} < \UB_{ls} - \LB_{ks}$  for each $(k,l,s)\in\combs_{ij}$.
\end{restatable}

We will present a computational proof of Theorem \ref{thm:RUideal} in Section \ref{sec:RUideal} as no analytic proof is known.

\subsection{Binary Embeddings}
\citet{StrongFloorLayout2017HuchetteVielma} also describe a formulation with less obvious interpretations. The $\disj_{ij}^4$ disjunction modeled by \ref{model:SU} has only four terms. By assigning each of its terms a distinct indicator vector $\boldsymbol{\bar\delta}_{kls} \in \{0,1\}^2$, the $\disj_{ij}^4$ disjunction can be modeled using only two binary indicator variables:

    \ParaTitle{The Simple Binary Formulation}{SB}\label{model:SB}
    \begin{NamedSubs}{SB}\begin{FeasRegion}
    c_{ls}  \GEQ  \LB_{ks} + \PM_{kls} - (\LB_{ks} + \PM_{kls} - \LB_{ls})\, \bcf{\boldsymbol{\bar\delta}_{kls},\boldsymbol{\delta}_{ij}}
        & \forall\ (k,l,s)\in\combs_{ij} \label{SB:lb}   \\
    c_{ks}  \LEQ  \UB_{ls} - \PM_{kls} - (\UB_{ls} - \PM_{kls} - \UB_{ks})\, \bcf{\boldsymbol{\bar\delta}_{kls},\boldsymbol{\delta}_{ij}}	 
        & \forall\ (k,l,s)\in\combs_{ij} \label{SB:ub}   \\   
    c_{ls} - c_{ks}  \GEQ  \PM_{kls} + (\LB_{ls} - \PM_{kls} - \UB_{ks})\, \bcf{\boldsymbol{\bar\delta}_{kls},\boldsymbol{\delta}_{ij}}
        & \forall\ (k,l,s)\in\combs_{ij} \label{SB:prec} \\
    c_{ks}  \in  [\LB_{ks},\UB_{ks}]
        & \forall\ (k,l,s)\in\combs_{ij} \label{SB:SBnds} \\
    \boldsymbol{\delta} \IN  \{0,1\}^2 \label{SB:indic}
    \end{FeasRegion}\end{NamedSubs}
where 
$\bcf:[0,1]^2 \times [0,1]^2 \to \mathbb{R}_+$ is a continuous approximation of the Inverse Boolean Comparison Function $\bcf^*\colon \{0,1\}^2 \times \{0,1\}^2 \to \{0,1\}$ defined
    \begin{equation}\label{BCF}
    \bcf^*{\mathbf{a},\mathbf{b}}  \EQ  [\mathbf{a}\boolcomp\mathbf{b}]  \EQ 
        \begin{cases}
        0, &\text{if }\mathbf{a}=\mathbf{b}, \\
        1, &\text{otherwise}.
        \end{cases}
    \end{equation}
    In particular, for $\bcf$ to produce valid inequalities, we require $\bcf^*(\mathbf{a}, \mathbf{b})\leq \bcf(\mathbf{a}, \mathbf{b}) \leq M \cdot \bcf^*(\mathbf{a}, \mathbf{b})$ for some value of $M\geq 1$ and all $\mathbf{a}, \mathbf{b} \in \{0,1\}^2$.
Note that \ref{model:SB} can be derived from \ref{model:SU} by replacing $\delta_{kls}$ with $1-\bcf{\boldsymbol{\bar\delta}_{kls},\boldsymbol{\delta}}$ and dropping constraint \ref{SU:disj}.

    \begin{table}[!ht]\centering
    \caption{An encoding based on two-digit Reflective Gray code alongside the reductions of $\bcf<\bar>$ and $\bcf<\tilde>$ for each of the codes.}
    \label{tab:G2Labeling}
    \begin{tabular}{c|c|c|c|c}
    $(k,l,s)$   & $\boldsymbol{\bar\delta}_{kls}$     & $\disj_{ij}$ term  &
        $\bcf<\bar>{\boldsymbol{\bar\delta},\boldsymbol{\delta}}$  &
        $\bcf<\tilde>{\boldsymbol{\bar\delta},\boldsymbol{\delta}}$  \\\hline
    $(i,j,x)$   & $(0,0)$                             & $i \preced{x} j$  &  $\delta_{ij} + \delta_{ji}$  &  $\delta_{ij} + \delta_{ji} - \Delta_{ij}$  \\
    $(i,j,y)$   & $(1,0)$                             & $i \preced{y} j$  &  $1 - \delta_{ij} + \delta_{ji}$  &  $1 - \delta_{ij} + \Delta_{ij}$  \\
    $(j,i,x)$   & $(1,1)$                             & $j \preced{x} i$  &  $2 - \delta_{ij} - \delta_{ji}$  &  $1 - \Delta_{ij}$  \\
    $(j,i,y)$   & $(0,1)$                             & $j \preced{y} i$  &  $1 + \delta_{ij} - \delta_{ji}$  &  $1 - \delta_{ji} + \Delta_{ij}$
    \end{tabular}
    \end{table}

Our implementations assign the indicator codes $\boldsymbol{\bar\delta}$ according to Table \ref{tab:G2Labeling}. The choice to use a reflective Gray code is not uniquely valid, but this particular order gives the basis of the FLP-SP formulation introduced in \cite{SequencePair2007MellerChenSherali} and lends itself well to the addition of some known symmetry-breaking inequalities which we will discuss in Section \ref{sec:ValidInequalities}. The FLP variant of this formulation is given in \cite[Section 5.2]{StrongFloorLayout2017HuchetteVielma}. The BLDP1 formulation from \cite[Eqns. (37)-(43)]{BlockLayout2005CastilloWesterlund} arises from a different assignment of indicator codes. Our first implementation of \ref{model:SB} remains linear by using the absolute difference function to approximate $\bcf^*$.

    \ParaTitle{The Simple Binary Linear Formulation}{SB-L}\label{model:SB-L}
    An implementation of \ref{model:SB} where $\bcf$ is actualized as
    \begin{align}\begin{split}
    \bcf<\bar>{\mathbf{a},\mathbf{b}}  &\EQ  \Vert\mathbf{a}-\mathbf{b}\Vert_1 
        \EQ  \Vert\mathbf{a}\Vert_1 - 2\mathbf{a}^\top\mathbf{b} + \Vert\mathbf{b}\Vert_1 \\
        &\EQ  a_1 + a_2 - 2a_1b_1 - 2a_2b_2 + b_1 + b_2. \label{BCF:Linear}
    \end{split}\end{align}
Letting $\boldsymbol{\delta} = (\delta_{ij},\delta_{ji})$, the reduction of $\bcf<\bar>{\boldsymbol{\bar\delta},\boldsymbol{\delta}}$ for each input code $\boldsymbol{\bar\delta}$ is given in Table \ref{tab:G2Labeling}.

By letting $\boldsymbol{\delta}_{ij} = (\delta_{ij},\delta_{ji})$ and assigning $\boldsymbol{\bar\delta}$ according to Table \ref{tab:G2Labeling}, constraint \eqref{SB:prec} reduces to
    \begin{subequations}
    \begin{align}
    c_{jx} - c_{ix}  &\GEQ  \PM_{ijx} + (\LB_{jx} - \PM_{ijx} - \UB_{ix})\,
        (\delta_{ij} + \delta_{ji}) 		\label{SB:iPjx}\\
    c_{jy} - c_{iy}  &\GEQ  \PM_{ijy} + (\LB_{jy} - \PM_{ijy} - \UB_{iy})\, 
        (1 - \delta_{ij} + \delta_{ji}) 	\label{SB:iPjy}\\
    c_{ix} - c_{jx}  &\GEQ  \PM_{jix} + (\LB_{ix} - \PM_{jix} - \UB_{jx})\, 
        (2 - \delta_{ij} - \delta_{ji}) 	\label{SB:jPix}\\
    c_{iy} - c_{jy}  &\GEQ  \PM_{jiy} + (\LB_{iy} - \PM_{jiy} - \UB_{jy})\, 
        (1 + \delta_{ij} - \delta_{ji}) 	\label{SB:jPiy}
    \end{align}
    \end{subequations}
As prescribed by Table \ref{tab:G2Labeling}, if $\boldsymbol{\delta} = (0,0)$, then \eqref{SB:iPjx} enforces $i \preced{x} j$; if $\boldsymbol{\delta} = (1,0)$, then \eqref{SB:iPjy} enforces $i \preced{y} j$; if $\boldsymbol{\delta} = (1,1)$, then \eqref{SB:jPix} enforces $j \preced{x} i$; and if $\boldsymbol{\delta} = (0,1)$, then \eqref{SB:jPiy} enforces $j \preced{y} i$. Since $\boldsymbol{\delta}$ must take one of these four vector-values, one of the terms of $\disj_{ij}^4$ will be enforced. Constraints \eqref{SB:lb} and \eqref{SB:ub} respectively behave like \eqref{SU:lb} and \eqref{SU:ub}; dynamically enforcing tight upper and lower bounds on the $\mathbf{c}$ variables. Despite being derivable from \ref{model:SU}, \ref{model:SB-L} does not inherit pairwise-idealness.

\begin{theorem}\label{thm:SBLNotIdeal}
\ref{model:SB-L} is \textit{not} pairwise-ideal in general.   
\end{theorem}
\begin{proof}
Consider the problem instance with two $2\times2$ objects with no clearances in a $10\times 10$ region so that $\LB_{1x},\LB_{2x},\LB_{1y},\LB_{2y} = 1$, $\UB_{1x},\UB_{2x},\UB_{1y},\UB_{2y} = 9$, and $\PM_{12x},\PM_{21x},\PM_{12y},\PM_{21y} = 2$. Notice that $\eqref{SB:lb}$, $\eqref{SB:ub}$, and $\eqref{SB:prec}$ with $\bcf \leftarrow \bar\bcf$, as in \ref{model:SB-L}, are given by
    $$
    \begin{bmatrix}
    0      & 1       & 0       & 0       & 2            & 2           \\
    0      & 0       & 0       & 1       & -2           & 2           \\
    -1     & 0       & 0       & 0       & 2            & 2           \\
    0      & 0       & -1      & 0       & -2           & 2           \\
    -1     & 1       & 0       & 1       & 10           & 10          \\
    0      & 0       & -1      & 1       & -10          & 10          
    \end{bmatrix}
    \begin{bmatrix}
    c_{1x} \\ c_{2x} \\ c_{1y} \\ c_{2y} \\ \delta_{12} \\ \delta_{21}
    \end{bmatrix}
    \GEQ
    \begin{bmatrix}
    3      \\ 1      \\ -7     \\ -9     \\ 2           \\ -8
    \end{bmatrix}
    $$
which is full rank and satisfied at equality by $c_{1x} = 9$, $c_{2x} = 1$, $c_{1y} = 9$, $c_{2y} = 1$, $\delta_{12} = \tfrac{1}{2}$, and $\delta_{21} = \tfrac{1}{2}$. So, the above solution is a fractional extreme point to the MBLP which implies that $\ref{model:SB-L}$ is not generally pairwise-ideal.
\end{proof}

The linear comparison function $\bcf<\bar>$ works well in practice, but has a significant flaw which precludes pairwise idealness: if neither entry in $\mathbf{b}$ matches its counterpart in and $\mathbf{a}$, then $\bcf<\bar>{\mathbf{a},\mathbf{b}} = 2$. For example:
    \begin{align*}
    \bcf<\bar>{\tvect{0}{0},\tvect{1}{1}}  &\EQ  \left\Vert\tvect{0}{0}\right\Vert_1 
        - 2\tvect{0}{0}^\top\tvect{1}{1} + \left\Vert\tvect{1}{1}\right\Vert_1 \\
        &\EQ  0 - 2(0) + 2  \EQ  2.
    \end{align*}
Consider the ramifications of this on \eqref{SB:prec}:
    \begin{align*}
    \PM_{kls} &+ 2(\LB_{ls} - \PM_{kls} - \UB_{ks}) \\
        &\EQ  \LB_{ls} - \UB_{ks} + (\LB_{ls} - \PM_{kls} - \UB_{ks}) \\
        &\EQ  \LB_{ls} - \UB_{ks} + (\m_{ls}^- + \tfrac{1}{2}\d_{ls} - \tfrac{1}{2}\d_{ks} - \max\{\m_{ks}^+,\m_{ls}^-\} - \tfrac{1}{2}\d_{ls} - \r_s + \m_{ks}^+ + \tfrac{1}{2}\d_{ks})\\
        &\EQ  \LB_{ls} - \UB_{ks} + (\m_{ls}^- + \m_{ks}^+ - \max\{\m_{ks}^+,\m_{ls}^-\} - \r_s)\\
        &\LEQ \LB_{ls} - \UB_{ks} + (\m_{ls}^- - \r_s)\\
        &\LS \LB_{ls} - \UB_{ks} \\
        &\LEQ c_{ls} - c_{ks}.
    \end{align*}
Thus, \eqref{SB:prec} is not facet defining since it is not tight to the convex hull where $\bcf<\bar>{\boldsymbol{\bar\delta},\boldsymbol{\delta}} = 2$. Similar failures can be shown for \eqref{SB:lb} and \eqref{SB:ub} under $\bcf\leftarrow\bcf<\bar>$. Hence the inclusion of the static bounds \eqref{SB:SBnds}.

Correcting this non-idealness is as simple as adjusting $\bcf<\bar>$ to subtract the cases where it takes the value two. Consider
    \begin{equation}
    \bcf<\tilde>{\mathbf{a},\mathbf{b}}  \EQ  \bcf<\bar>{\mathbf{a},\mathbf{b}} - \vert a_1-b_1\vert\vert a_2-b_2\vert.
    \end{equation}
A value of 1 is subtracted if and only if $[\mathbf{a} - \mathbf{b}] = \boldsymbol{1}$. By exploiting its binary domain, $\bcf<\tilde>$ can be written in a closed form which is multilinear in $\mathbf{b}$:
    \begin{align}\begin{split}
    \bcf<\tilde>{\mathbf{a},\mathbf{b}}  &\EQ  b_1 + b_2 - b_1b_2 + a_1(1-a_2)(1 - 2b_1 - b_2 + 2b_1b_2) \\
        &\hspace{3em} + (1-a_1)a_2(1 - b_1 - 2b_2 + 2b_1b_2) + a_1a_2(1 - b_1 - b_2).
    \end{split}\end{align}
By introducing the auxiliary variable $\Delta_{ij}$ and applying the McCormick envelope of $\Delta_{ij} = \delta_{ij}\delta_{ji}$, it is possible to represent $\bcf<\tilde>$ linearly.

    \ParaTitle{The Simple Binary Multilinear Formulation}{SB-M}\label{model:SB-M}\
    An implementation of \ref{model:SB} where: $\bcf$ is actualized as
    \begin{align}
    \begin{split}
    \bcf<\tilde>{\boldsymbol{\bar\delta},\boldsymbol{\delta}}  &\EQ  \delta_{ij} + \delta_{ji} - \Delta_{ij} + \bar\delta_1(1-\bar\delta_2)(1 - 2\delta_{ij} - \delta_{ji} + 2\Delta_{ij}) \\
        &\hspace{3em} + (1-\bar\delta_1)\bar\delta_2(1 - \delta_{ij} - 2\delta_{ji} + 2\Delta_{ij}) + \bar\delta_1\bar\delta_2(1 - \delta_{ij} - \delta_{ji}),
    \end{split}
    \end{align}
    and the following constraints on the variable $\Delta_{ij}$ are added
    \begin{NamedSubs}{SB-M}\begin{FeasRegion}[align]
    \delta_{ij} + \delta_{ji} - \Delta_{ij}  &\LEQ  1 \\
                  \delta_{ij} - \Delta_{ij}  &\GEQ  0 \\
                  \delta_{ji} - \Delta_{ij}  &\GEQ  0 \\
                                \Delta_{ij}  &\GEQ  0,
    \end{FeasRegion}\end{NamedSubs}
    and the static bounds \eqref{SB:SBnds} on $\mathbf{c}$ are dropped.

\vspace{1em}\noindent Letting $\boldsymbol{\delta} = (\delta_{ij},\delta_{ji})$, the reduction of $\bcf<\tilde>{\boldsymbol{\bar\delta},\boldsymbol{\delta}}$ for each input code $\boldsymbol{\bar\delta}$ is given in Table \ref{tab:G2Labeling}.
    
\begin{restatable}{theorem}{SBMideal}\label{thm:SBMideal}
\ref{model:SB-M} is pairwise-ideal if $\PM_{kls} < \UB_{ls} - \LB_{ks}$ for each $(k,l,s) \in \combs_{ij}$.
\end{restatable}

We will present a computational proof of Theorem \ref{thm:SBMideal} in Section \ref{sec:SBMideal} as no analytic proof is known.

    \begin{table}[!ht]\centering
        \caption{The sizes of each formulation with respect to the number of pairs $\binom{n}{2}$ of objects to be packed. An asterisk $*$ indicates that the formulation is known to be pairwise ideal under some widely applicable conditions.}
        \label{tab:FormSizes}
        \begin{tabular}{|c||c|c|c||c|c||c|}
        \hline \multirow{2}{*}{Formulation} & \multicolumn{3}{c||}{Constraints} & \multicolumn{2}{c||}{Auxiliary Variables} & Disjunctive      \\\cline{2-6}
                         & Precedence      & Bounds          & Logic           & Continuous & Binary          & Terms           \\\hline
        \ref{model:SU}*   & $4$  & $8$  & $1$  & $0$  & $4$ & $4$ \\ 
        \ref{model:RU}*   & $4$  & $8$  & $3$  & $0$  & $4$ & $8$ \\
        \ref{model:SB-L}  & $4$  & $8$  & $0$  & $0$  & $2$ & $4$ \\
        \ref{model:SB-M}* & $4$  & $8$  & $3$  & $1$  & $2$ & $4$ \\
        \hline\end{tabular}
    \end{table}

\subsection{Computational Results}\label{sec:CutsTest}
In this section, we report the results of several computational experiments concerning the Unary and Binary formulations. All tests were run in Gurobi through the Python API on a Dell Precision 5820 X-Series workstation with 64 gigabytes of RAM, an AMD Radeon Pro WX2100, and an 18-core Intel Core i9-10980XE running at three gigahertz.

We chose \hyperlink{https://sites.google.com/gcloud.fe.up.pt/ cutting-and-packing-tools/2dcpackgen}{2DCPackGen} \cite{2DCPackGen2014SilvaOliveiraWascher}, a random instance generator for a variety of two-dimensional cutting and packing problems, to generate strip packing instances. All problems feature a 100-unit wide strip and objects between 5 and 30 units in either direction according to a random sampling of the Beta(2,5) distribution (the ``small and square" characteristic). Instances were generated with each of $\numobjs = 10, 15, 20$, and $25$ before clearances were added. Each side of each object has a 50\% chance to be given a clearance according to a uniform distribution between zero and the dimension of the parent object. Our problem instances are available, alongside all the code used to generate and solve them, at \hyperlink{https://github.com/jfravel/Ideal-O-Matic}{https://github.com/jfravel/Ideal-O-Matic}.

While Gurobi usually finds initial feasible solutions to the Strip Packing Problem quickly, they are often relatively sparse and result in poor initial estimations of the optimal height. We implemented a greedy heuristic to generate better initial solutions: (1) Order the objects in increasing order of height ($\m_{iy}^- + \d_{iy} + \m_{iy}^+$); (2) Pack objects in a row along the base of the strip until adding another object would overflow the bounds; (3) given the tallest physical object and tallest clearance of the first row, pack objects in a row above the first until adding another object would overflow the bounds while making sure that no occlusion occurs between the rows; and (4) repeat step (3) until all object have been packed. See Figure \ref{fig:InitialSolutions} for an example of Gurobi's initial solutions and those generated greedily. Clearly, more improvement could be made with techniques like those described in \cite{GeneticStripPacking2006Bortfeldt}, \cite{ApplicationFixedLayout2021McKendallHakobyan}, or \cite{FireflyFixedLayout2017IngoleSingh} but our focus is on exact methods and this greedy heuristic is a simple way to ensure that each of the formulations start from an equitable initial solution.

    \begin{figure}[!ht]\centering
    \begin{subfigure}{.3\textwidth}\centering
    \includegraphics[height=.25\textheight]{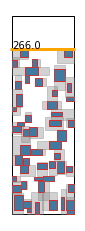}
    \caption{Gurobi's solution starts with an optimality gap of 84.6\%.}
    \end{subfigure}\hspace{.5in}
    \begin{subfigure}{.3\textwidth}\centering
    \includegraphics[height=.25\textheight]{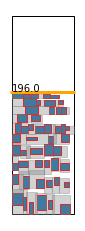}
    \caption{The greedy solution starts with an optimality gap of 79.1\%.}
    \end{subfigure}
    \caption{Two possible initial solutions to the $\numobjs=50$ Strip Packing Problem.}
    \label{fig:InitialSolutions}
    \end{figure}

Our first experiment involves solving each of the strip packing problems under each of the formulations with each combination of Gurobi's Cuts and Heuristics parameters on or off. The results can be found in Table \ref{tab:CutsImprove} below and Table \ref{tab:CutsTest} in the appendix. Notice that both cuts and heuristics seem to have a significant negative effect on the convergence of the IP; we will discuss this further in Section \ref{sec:Cuts}.

    \begin{table}[!ht]\centering
    \caption{The percent improvement seen in the \ref{model:SU}, \ref{model:RU}, and \ref{model:SB} formulations when turning off cuts or heuristics in terms of convergence and nodes visited. Positive value indicates a reduction in the metric while negative values indicate worse performance. An asterisk (*) indicates that the model converged with the given parameters where it did not under default settings. Columns C and H indicate if cuts and heuristics, respectively, were active (\checkmark) or disabled (\xmark). More detailed results are given in Table \ref{tab:CutsTest} in Appendix \ref{sec:AdditionalTablesAndPlots}.}
    \label{tab:CutsImprove}
    \setlength{\tabcolsep}{4pt}

\begin{tabular}{|ccc||cc||cc||cc||cc||}
\cline{4-11}\multicolumn{1}{c}{} &       &       & \multicolumn{2}{c||}{SU} & \multicolumn{2}{c||}{RU} & \multicolumn{2}{c||}{SB-L} & \multicolumn{2}{c||}{SB-M} \\
\hline
N     & C     & H     & Time/Gap & Nodes & Time/Gap & Nodes & Time/Gap & Nodes & Time/Gap & Nodes \\
\hline
\multirow{3}[2]{*}{10} & \cmark & \xmark & 95\%  & 42\%  & 90\%  & -42\% & 86\%  & -37\% & 76\%  & -110\% \\
      & \xmark & \cmark & 75\%  & 26\%  & 21\%  & -61\% & 64\%  & 4\%   & 43\%  & -25\% \\
      & \xmark & \xmark & 97\%  & 40\%  & 95\%  & -5\%  & 95\%  & 23\%  & 92\%  & -46\% \\
\hline
\multirow{3}[2]{*}{15} & \cmark & \xmark & 0\%   & -36\% & -194\% & -12\% & -33\% & -68\% & 11\%  & -103\% \\
      & \xmark & \cmark & 40\%  & -230\% & *     & -128\% & 33\%  & -268\% & 33\%  & -78\% \\
      & \xmark & \xmark & 40\%  & -1154\% & *     & -521\% & 33\%  & -2756\% & 22\%  & -355\% \\
\hline
\multirow{3}[2]{*}{20} & \cmark & \xmark & 15\%  & -588\% & -9\%  & -128\% & -3\%  & -214\% & -2\%  & -170\% \\
      & \xmark & \cmark & 17\%  & -287\% & 0\%   & -83\% & 4\%   & -109\% & -2\%  & -33\% \\
      & \xmark & \xmark & 24\%  & -3641\% & 9\%   & -1749\% & 8\%   & -1126\% & 6\%   & -511\% \\
\hline
\multirow{3}[2]{*}{25} & \cmark & \xmark & 5\%   & -326\% & -1\%  & -184\% & 4\%   & -377\% & -2\%  & -208\% \\
      & \xmark & \cmark & 1\%   & -218\% & 8\%   & -67\% & -1\%  & -185\% & 0\%   & -47\% \\
      & \xmark & \xmark & 21\%  & -2867\% & 21\%  & -1490\% & 7\%   & -1517\% & -1\%  & -470\% \\
\hline
\end{tabular}%

    \end{table}

\section{Model Enhancements}\label{sec:ModelVariations}
In this section, we describe several modifications which can be made to the above formulations. These include symmetry breaking inequalities and branching priorities. We perform a computational experiment on these variations before proceeding.

\subsection{Static Bounds}\label{sec:StaticBounds} 
Notice that constraints \eqref{SU:lb}, \eqref{SU:ub}, \eqref{RU:lb}, \eqref{RU:ub}, \eqref{SB:lb}, and \eqref{SB:ub} place upper and lower bounds on the $\mathbf{c}$ variables which vary according to $\boldsymbol{\delta}$. We include these dynamic bounds in keeping with~\cite{StrongFloorLayout2017HuchetteVielma} despite their non-necessity. More sparse, valid formulations can be constructed by dropping each of these and replacing them with the static bounds $\LB_{is} \leq c_{is} \leq \UB_{is}$ for each object $i\in\setobjs$ and each direction $s\in\{x,y\}$. In fact, doing this does not affect the pairwise idealness of the parent formulation. We experiment with static bounds for the sake of completeness.

\subsection{Symmetry Breaking Inequalities}\label{sec:ValidInequalities}
Consider a set of three objects $i$, $j$, and $k$ and notice that if $i\preced{s}j$ and $j\preced{s}k$, then it must also be that $i\preced{s}k$. This relationship can be enforced in \ref{model:SU} and \ref{model:RU} by adding the inequality $\delta_{ijs} + \delta_{jks} - \delta_{iks}  \LEQ  1$ and holds for any permutation of the triple:

    \ParaTitle{The Unary Sequence-Pair Inequalities}{SPU}\label{cuts:SPU}
    \begin{FeasRegion}*
    \delta_{ijs} + \delta_{jks} - \delta_{iks}  \LEQ  1     & \forall\ (i,j,k)\in \mathscr{T},\ s\in\{x,y\}
    \end{FeasRegion}

where $\mathscr{T}:= \{i,j,k \in \setobjs: i < j < k\}$ is the set of distinct triples of object indices. This can be derived by applying the lower McCormick envelope to  $\delta_{iks} \geq \delta_{ijs}\delta_{jks}$ for all triples $(i,j,k)$ of objects.

The indicator variables lose their obvious interpretations in the binary formulations, but similar sequence inequalities are still possible. Described in \cite[Appendix D.2]{StrongFloorLayout2017HuchetteVielma} but originally from \cite{SequencePair2007MellerChenSherali}, the following inequalities are valid for any three objects $\object_i$, $\object_j$, and $\object_k$ under \ref{model:SB}:

    \ParaTitle{The Binary Sequence-Pair Inequalities}{SPB}\label{cuts:SPB}
    \begin{FeasRegion}*
    0  \LEQ  \delta_{ij} + \delta_{jk} - \delta_{ik}  \LEQ  1 	
        & \forall\ (i,j,k)\in\mathscr{T} \nonumber\\
    0  \LEQ  \delta_{ji} + \delta_{kj} - \delta_{ki}  \LEQ  1 	
        & \forall\ (i,j,k)\in\mathscr{T} \nonumber
    \end{FeasRegion}
\noindent  
These inequalities do not remove any feasible solutions in the space of $\mathbf{c}$ variables, but they can reduce symmetry in the space of $\boldsymbol{\delta}$ variables. See Figure \ref{fig:SSBsymetry} for an example.

    \begin{figure}[!ht]\centering
        \begin{tikzpicture}[font=\Large]
        \draw[line width=1pt] (4,2.5) rectangle (6,4.5);
        \node at (5,3.5) {$\object_j$};
        \draw[line width=1pt] (2,1.25) rectangle (4,3.25);
        \node at (3,2.25) {$\object_i$};
        \draw[line width=1pt] (0,0) rectangle (2,2);
        \node at (1,1) {$\object_k$};
        \end{tikzpicture}
    \caption{A symmetric solution. It is clear that $i\preced{x}j$, $k\preced{y}j$, and $k\preced{x}i$ which implies an indicator solution of $\boldsymbol{\delta}_{ij} = (0,0)$, $\boldsymbol{\delta}_{jk} = (0,1)$, and $\boldsymbol{\delta}_{ik} = (1,1)$. However, this violates the first inequality of \ref{cuts:SPB}. Instead, it is necessary to pick the alternative $k\preced{x}j$ so that $\boldsymbol{\delta}_{jk} = (1,1)$.}
    \label{fig:SSBsymetry}
    \end{figure}
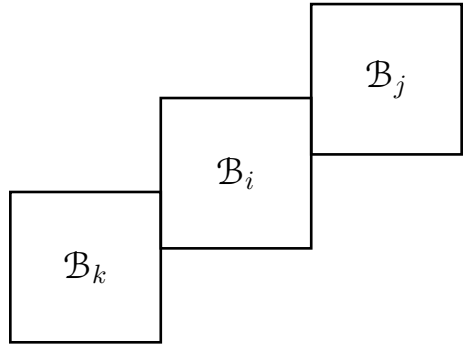
    
The Sequence Pair Inequalities do not affect the pairwise idealness of the parent formulation since their implementation requires at least three objects.

\subsection{Cuts and Branching}\label{sec:Cuts}
According to experimental results in given in Table \ref{tab:CutsImprove}, leaving Gurobi's cuts and heuristics parameters active seriously hinders the convergence rate of \ref{model:SU} and \ref{model:RU} in particular. Figure \ref{fig:CutsTestPlots} plots the convergence of these formulations, varying Gurobi's cuts and heuristics parameters, for $\numobjs = 25$ (the same runs from Table \ref{tab:CutsImprove}). Plotted adjacent is the number of iterations per node according to Gurobi's log files. Notice that the number of iterations per node appears to be strongly correlated with the rate of convergence; in particular, both \ref{model:SU} and \ref{model:RU} show the fastest convergence and lowest iterations per node with cuts and heuristics both deactivated. 
    \begin{figure}[!ht]\centering
    \begin{subfigure}[b]{0.85\textwidth}\centering
        \includegraphics[width=\textwidth]{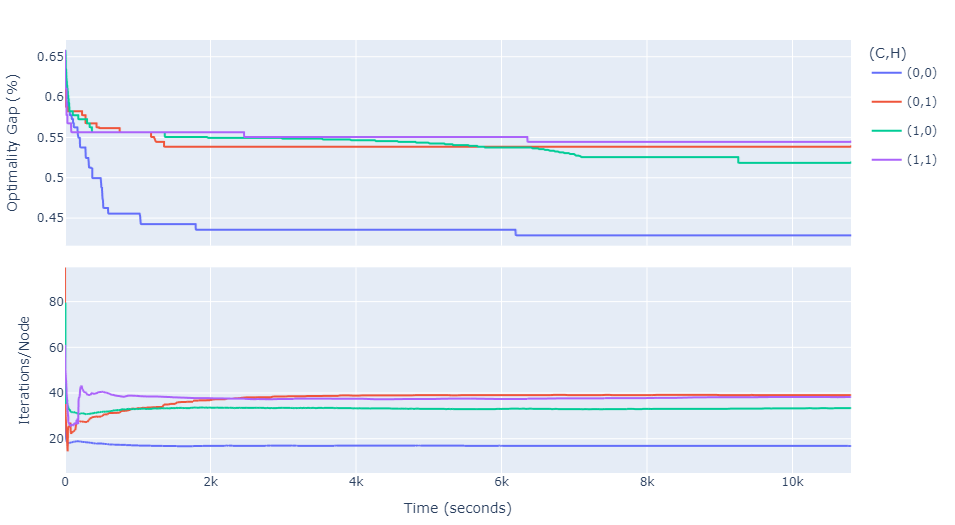}
        \caption{Cuts and Heuristics test for \ref{model:SU} Strip Packing with $\numobjs = 25$.}
        \label{fig:CutsTestPlot:SU25}
    \end{subfigure}
    \begin{subfigure}[b]{0.85\textwidth}\centering
        \includegraphics[width=\textwidth]{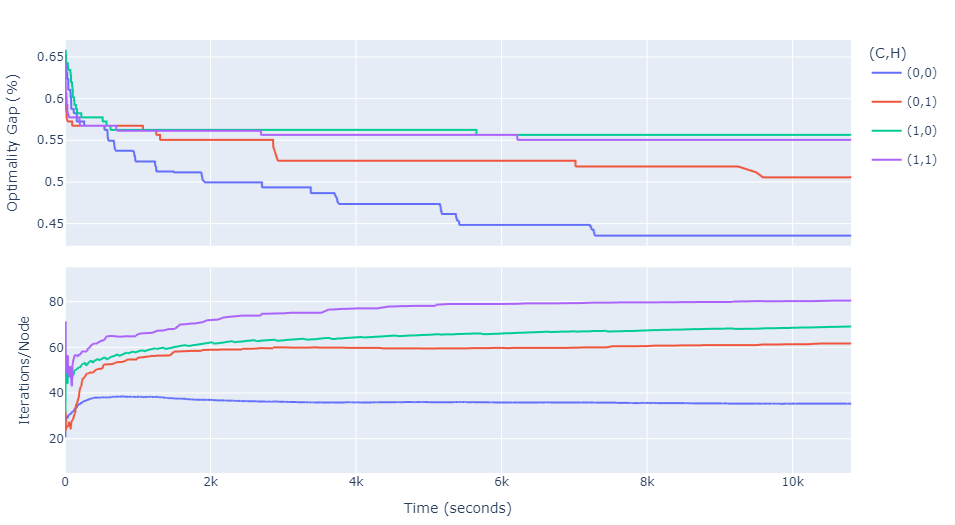}
        \caption{Cuts and Heuristics test for \ref{model:RU} Strip Packing with $\numobjs = 25$.}
        \label{fig:CutsTestPlot:RU25}
    \end{subfigure}
    \caption{Plots comparing each of the four configurations on the 25-object instance for \ref{model:SU} and \ref{model:RU} in the Cuts and Heuristics experiments (see Table \ref{tab:CutsTest}). We experiment with Gurobi's Cuts and Heuristics Parameters: Cuts are active where $C = 1$ and Heuristics are active where $H = 1$. Gaps and Iterations per Node are extracted from Gurobi's output log files via the gurobi-logtools \cite{gurobi_logtools} package for Python.}
    \label{fig:CutsTestPlots}
    \end{figure}

Cuts (and heuristic) actually seem to increase the number of simplex iterations spent at each node. This effect is less significant, but still present, with \ref{model:SB-L} and \ref{model:SB-M}; see Figure \ref{fig:CutsTestPlotsSB} in Appendix \ref{sec:AdditionalTablesAndPlots}.

Given the poor performance of cuts (and heuristics) on this problem, we implement a branching scheme inspired by the work of Sherali, Fraticelli, and Meller~\cite{EnhancedFormulations2003SheraliFraticelliMeller} in hopes of improving the performance of pure branch-and-bound. The position of a large object will have a more noticeable, global impact on the indicator variables than will the position of a small object. In particular, the precedence between a \textit{pair} of large objects ought to be established early in the branch-and-bound tree. For this reason, we assign each indicator $\delta_{ijs}$ a branching priority given by: 
    \begin{equation}\label{eq:BranchPriorityUnary}
    \min\{\m_{is},\m_{js}\} + \left(\max_{k\in\setobjs}\{\m_{ks}\} + 1\right)\left(\min\{\d_{is},\d_{js}\} + \left(\max_{k\in\setobjs}\{\d_{ks}\} + 1\right)\min(a_i,a_j)\right)
    \end{equation}
where $\m_{is} = \m_{is}^++m_{is}^-$ and $a_i = \d_{ix}\d_{iy}$ for each $i\in\setobjs$. This expression places emphasis, first, on physical area; followed by physical size in the relevant direction; and lastly, on the combined size of clearances in the relevant direction.

Since the indicator variables $\boldsymbol{\delta}_{ij} = (\delta_{ij},\delta_{ji})$  do not have a clear connection to the direction $s\in\{x,y\}$, we drop direction from the branching priority. Thus, both $\delta_{ij}$ and $\delta_{ji}$ are given a branching priority value according to:
    \begin{equation}\label{eq:BranchPriorityBinary}
    \min\{\m_{i},\m_{j}\} + \left(\max_{k\in\setobjs}\{\m_{k}\} + 1\right)\min(a_i,a_j)
    \end{equation}
where $\m_{i} = \m_{ix}+\m_{iy}$ and $a_i = \d_{ix}\d_{iy}$ for each $i\in\setobjs$. In either case, Gurobi is now more likely to branch on precedence indicators between pairs of large objects with large clearances. Naturally, adding a priority metric to the indicator variables does not affect the pairwise-idealness of the parent formulation.

\subsection{Computational Results}
Our second experiment involves solving the strip packing problems under each of the formulations and each variant with Gurobi's Cuts and Heuristics turned off. The results can be found in Table \ref{tab:SBSPImprove} below and Table \ref{tab:SBSPTest} in Appendix \ref{sec:AdditionalTablesAndPlots}. Notice that the binary formulation \ref{model:SB-L} and \ref{model:SB-M} see significant time improvement from both the branching priority and the sequence pair inequalities. Equally impressive is the dramatic improvement in node count offered by the Sequence Pair inequalities in each setting while usually providing comparable or improved convergence. On the other hand, the static bounds offer inconsistent performance. 

    \begin{table}[!ht]\centering
    \caption{The percent improvement each variation gives over standard in the \ref{model:SU}, \ref{model:RU}, and \ref{model:SB} formulations in terms of convergence and nodes visited. Gurobi's Cuts and Heuristics are turned off in each case.  A positive value indicates a reduction in the metric and better performance. An asterisk (*) indicates that the model converged with the given parameters where it did not under default settings, while a circle ($\circ$) indicates the inverse. ``Static'' indicates that the first two constraints of each formulation have been replaced with the static bounds $c_{is} \in [\LB_{is},\UB_{is}]$; ``Branch'' indicates that the branching priorities \eqref{eq:BranchPriorityUnary} or \eqref{eq:BranchPriorityBinary} have been applied to the binary variable; and ``Sequence'' indicates that the Sequence-Pair inequalities \ref{cuts:SPU} or \ref{cuts:SPB} have been added to the formulation. More detailed results are given in Table \ref{tab:SBSPTest} in the Appendix.}
    \label{tab:SBSPImprove}
    \setlength{\tabcolsep}{4pt}
\begin{tabular}{|cc||cc||cc||cc||cc||}
\cline{3-10}\multicolumn{1}{c}{} &       & \multicolumn{2}{c||}{SU} & \multicolumn{2}{c||}{RU} & \multicolumn{2}{c||}{SB-L} & \multicolumn{2}{c||}{SB-M} \\
\multicolumn{1}{c}{} &       & Time/Gap & Nodes & Time/Gap & Nodes & Time/Gap & Nodes & Time/Gap & Nodes \\
\hline
\multirow{4}[2]{*}{10} & Static & 0\%   & -31\% & -1022\% & -4\%  & 41\%  & 8\%   & 9\%   & -47\% \\
      & Branch & -108\% & -198\% & -23\% & -80\% & 12\%  & -70\% & 3\%   & 0\% \\
      & Sequence & -134\% & -64\% & -106\% & -28\% & -47\% & 1\%   & -75\% & -42\% \\
      & Seq \& Br & -43\% & 27\%  & -31\% & 11\%  & -32\% & -30\% & 9\%   & -8\% \\
\hline
\multirow{4}[2]{*}{15} & Static & 2\%   & -736\% & 18\%  & -34\% & 49\%  & -620\% & 29\%  & -266\% \\
      & Branch & -161\% & -65\% & $\circ$ & -67\% & 25\%  & -77\% & 14\%  & -104\% \\
      & Sequence & *     & 77\%  & 57\%  & 46\%  & *     & 100\% & *     & 99\% \\
      & Seq \& Br & 0\%   & -89\% & $\circ$ & -116\% & *     & 100\% & *     & 100\% \\
\hline
\multirow{4}[2]{*}{20} & Static & -18\% & -46\% & -22\% & 33\%  & -2\%  & -146\% & 7\%   & -223\% \\
      & Branch & -7\%  & 3\%   & 1\%   & -6\%  & 16\%  & -72\% & 5\%   & -54\% \\
      & Sequence & -10\% & 90\%  & -18\% & 85\%  & 29\%  & 39\%  & 37\%  & 46\% \\
      & Seq \& Br & 0\%   & 89\%  & 0\%   & 81\%  & 33\%  & 42\%  & 44\%  & 51\% \\
\hline
\multirow{4}[1]{*}{25} & Static & -24\% & -31\% & -23\% & 29\%  & 5\%   & -231\% & -1\%  & -257\% \\
      & Branch & -3\%  & 2\%   & 0\%   & -12\% & 20\%  & -107\% & 0\%   & -55\% \\
      & Sequence & -14\% & 91\%  & -18\% & 84\%  & 26\%  & 41\%  & 23\%  & 40\% \\
      & Seq \& Br & -7\%  & 90\%  & -9\%  & 82\%  & 28\%  & 50\%  & 35\%  & 45\% \\
\hline
\end{tabular}%

    \end{table}

    \begin{figure}[!ht]\centering
    \begin{subfigure}[b]{0.85\textwidth}\centering
        \includegraphics[width=\textwidth]{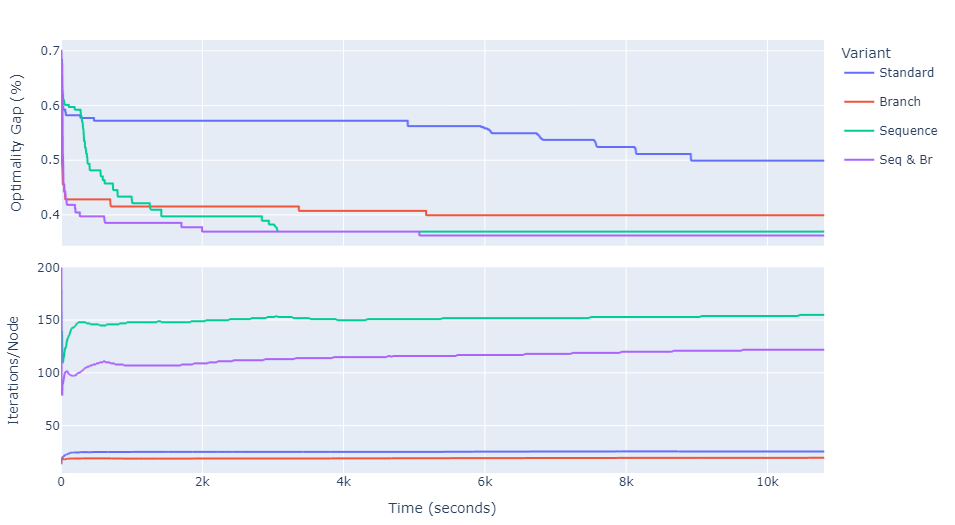}
        \caption{Sequence-Pair and Branching Priority test for \ref{model:SB-L} Strip Packing with $\numobjs = 25$.}
        \label{fig:SeqAndBranch:SBL25}
    \end{subfigure}
    \begin{subfigure}[b]{0.85\textwidth}\centering
        \includegraphics[width=\textwidth]{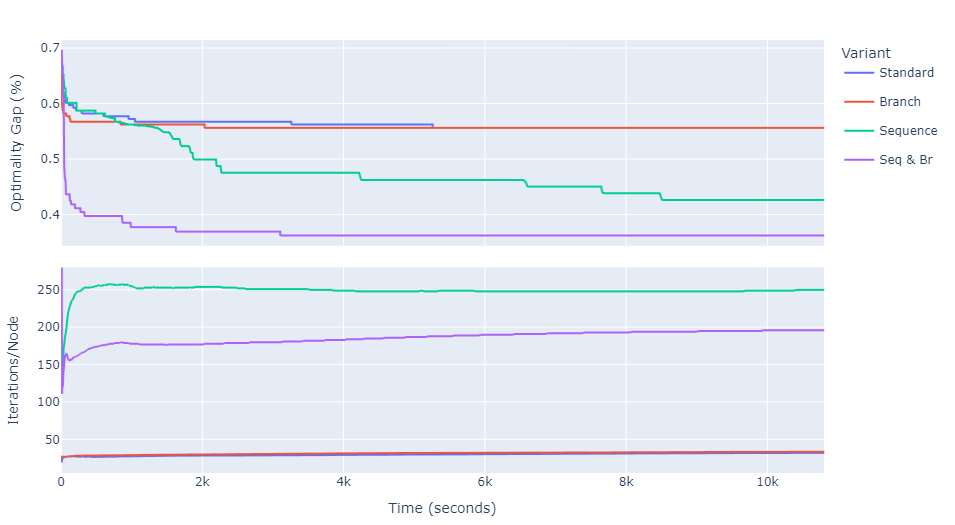}
        \caption{Sequence-Pair and Branching Priority test for \ref{model:SB-M} Strip Packing with $\numobjs = 25$.}
        \label{fig:SeqAndBranch:SBM25}
    \end{subfigure}
    \caption{Plots comparing each of the four configurations on the 25-object instance for \ref{model:SB} in the Sequence Pair and Branch Priority experiments (see Table \ref{tab:SBSPTest}). We experiment with adding symmetry-breaking inequalities \eqref{cuts:SPB} and branching priority \eqref{eq:BranchPriorityBinary}. Gaps and Iterations per Node are extracted from Gurobi's output log files via the \protect\hyperlink{https://github.com/Gurobi/gurobi-logtools}{gurobi-logtools}\cite{gurobi_logtools} package for Python.}
    \label{fig:SBSPImprovePlots}
    \end{figure}

\section{The Ideal O'Matic: Automating proofs of Idealness}\label{sec:IdealOMatic}
The analytic proof of Theorem \ref{thm:SUideal} is exceeding long while no analytic proof is known for Theorem \ref{thm:RUideal} or \ref{thm:SBMideal}; it would be nice to automate the process. Consider a general Mixed-Binary Linear Program (MBLP):
    \begin{LPArray}*[align]{Minimize}{\mathbf{c^\top\tvect{\mathbf{x}}{\mathbf{y}}}}
	& A\tvect{\mathbf{x}}{\mathbf{y}}  &\geq  &\mathbf{b}  \\
	& 	\mathbf{y}  &\in   &\{0,1\}^n
    \end{LPArray}
where $A\in\mathbb{Q}^{k\times(m+n)}$, $\mathbf{b}\in\mathbb{Q}^k$, and $\mathbf{c}\in\mathbb{Q}^{m+n}$. That is, the MBLP has $k$ constraints, $m$ continuous variables, and $n$ binary variables. 

Recall that any solution to the continuous relaxation of the MBLP that satisfies a linearly independent set of $m + n$ constraints at equality is called an extreme point. The MBPL is \textit{ideal} if and only if each such extreme point also satisfies the integer constraints. The forthcomming scheme demonstrates idealness by searching for non-integer extreme points; if no such point can be found, then the MBPL is ideal.

\subsection{An MBLP for Demonstrating MBPL Idealness} The ``integrality'' of a given solution is measured by the penalty function $\phi(y) = 1 - \big\vert 2y - 1 \big\vert$ which will take the value zero on integer solutions but positive value otherwise (see Figure \ref{fig:PenFunc}). If the relaxed MBLP has even one extreme point with even one $y$ variable for which $\phi(y) > 0$, then it is not an ideal formulation. 

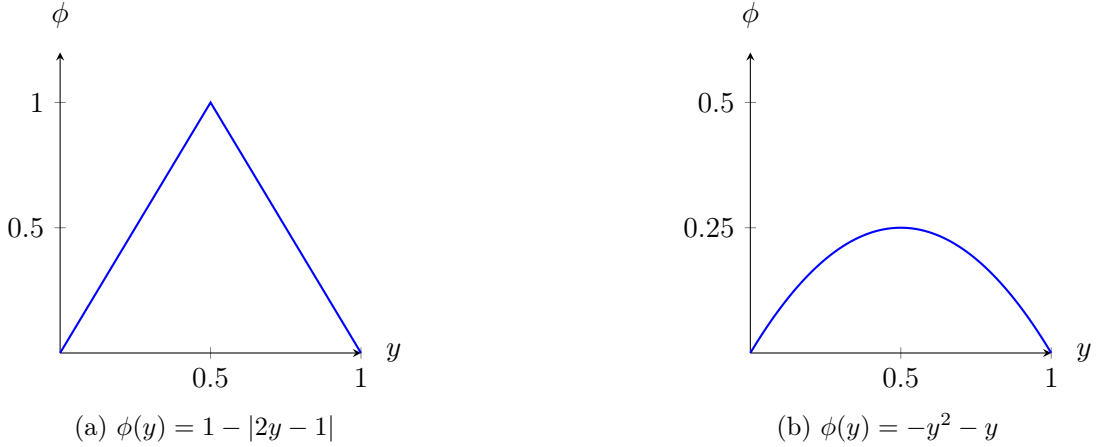
\begin{figure}[ht]
    \centering
    \begin{subfigure}[b]{0.45\textwidth}
        \centering
        \begin{tikzpicture}
		\begin{axis}[
		    axis lines=middle,
		    clip=false,
		    ymin=0, ymax=1.2,
		    xmin=0, xmax=1,
		    domain=0:1,
		    xlabel=\( y \),
		    ylabel=\( \phi \),
		    xtick={0,0.5,1},
		    ytick={0,0.5,1},
		    xticklabels={0,0.5,1},
		    yticklabels={0,0.5,1},
		    every axis x label/.style={
		        at={(ticklabel* cs:1.05)},
		        anchor=west,
		    },
		    every axis y label/.style={
		        at={(ticklabel* cs:1.05)},
		        anchor=south,
		    },
		    width=4cm, %
		    height=4cm, %
		    scale only axis, %
		]
		
		\addplot+[sharp plot, no marks, thick, blue] coordinates {(0,0) (0.5,1) (1,0)};
		
		\end{axis}
	\end{tikzpicture}
        \caption{$\phi(y) = 1 - |2y - 1|$}
        \label{PenFunc:Abs}
    \end{subfigure}
    \hfill %
    \begin{subfigure}[b]{0.45\textwidth}
        \centering
        \begin{tikzpicture}
		\begin{axis}[
		    axis lines=middle,
		    clip=false,
		    ymin=0, ymax=0.6,
		    xmin=0, xmax=1,
		    domain=0:1,
		    xlabel=\( y \),
		    ylabel=\( \phi \),
		    xtick={0,0.5,1},
		    ytick={0,0.25, 0.5},
		    xticklabels={0,0.5,1},
		    yticklabels={0,0.25, 0.5},
		    every axis x label/.style={
		        at={(ticklabel* cs:1.05)},
		        anchor=west,
		    },
		    every axis y label/.style={
		        at={(ticklabel* cs:1.05)},
		        anchor=south,
		    },
		    width=4cm, %
		    height=4cm, %
		    scale only axis, %
		]
		\addplot+[no marks, thick, blue, domain=0:1, samples=100] {-x^2 + x};
		\end{axis}
		\end{tikzpicture}
        \caption{$\phi(y) = -y^2 - y$}
        \label{PenFunc:Quad}
    \end{subfigure}
    
    \caption{Some options for the penalty function $\phi$.  We use (\subref{PenFunc:Abs}) as it can be encoded with linear constraints, whereas the quadratic version (\subref{PenFunc:Quad}) introduces unnecessary nonlinearities. }
    \label{fig:PenFunc}
	\end{figure}

The continuous relaxation of our MBLP is given by:
    \begin{center}
    \begin{minipage}{0.30\textwidth}
        \begin{LPArray}*[align]{Minimize}{\mathbf{c^\top\tvect{\mathbf{x}}{\mathbf{y}}}}
         & A\tvect{\mathbf{x}}{\mathbf{y}}  &\geq  &\mathbf{b}  \\
         &  					\mathbf{y}  &\leq  &\mathbf{0}  \\
         &  					\mathbf{y}  &\geq  &\mathbf{1}
        \end{LPArray}
    \end{minipage}
    $\qquad\rightarrow\qquad$
    \begin{minipage}{0.40\textwidth}
        \begin{subequations}\label{model:MBLP-CR}
            \begin{LPArray}[align]{Minimize}{\mathbf{c^\top\tvect{\mathbf{x}}{\mathbf{y}}}  \label{MBLP-CR:Obj}}
             & A'\tvect{\mathbf{x}}{\mathbf{y}}  &\geq  &\mathbf{b}'				 \label{MBLP-CR:Feas}.
            \end{LPArray}
        \end{subequations}
    \end{minipage}
    \end{center}
It is important to remember that bounds on the integer variables exist as they can contribute to the $m+n$ tight constraints. However, such simple inequalities are often obviously redundant to other constraints; for example, equation \eqref{SU:disj} automatically implies $\delta_{kls} \leq 1$ when composed with non-negativity from relaxing \eqref{SU:indic}. Since we will ultimately be looking for linear independence, discarding redundant constraints can improve performance considerably. For the sake of notation, we compose the non-redundant bounds on $\mathbf{y}$ into the matrix $A$, call this new matrix $A' \in \mathbb{R}^{k'\times(m+n)}$.

To disprove the existence of fractional extreme points, we maximize the penalty function $\Phi(\mathbf{y}) = \sum_{i=1}^n\phi(y_i)$ over the set of points which are feasible to the MBLP and tight to $m+n$ linearly independent constraints:

    \ParaTitle{Ideal O'Matic}{IOM}\label{model:IOM}
    \begin{NamedSubs}{IOM}\begin{LPArray}[align]{Maximize}{\Phi(\mathbf{y}) = \boldsymbol{1}^\top\boldsymbol{\phi}   \label{IOM:Obj}}
	& \boldsymbol{\phi}  &\leq  &2\mathbf{y}
            \label{IOM:Pen+}   \\
	  & \boldsymbol{\phi}  &\leq  &2 - 2\mathbf{y}
            \label{IOM:Pen-}   \\
	  & A'\tvect{\mathbf{x}}{\mathbf{y}}  &\geq  &\mathbf{b}'
            \label{IOM:Feas}   \\
	  & A'\tvect{\mathbf{x}}{\mathbf{y}}  &\leq  &\mathbf{b}' + M(1-\boldsymbol{\eta})
            \label{IOM:Tight}  \\
	  &	\sum_{i=1}^k\eta_i  &=     &m + n
            \label{IOM:Extrm}  \\
	  &	\sum_{i \in I} \eta_i &\leq  & |I|-1     
	       & \forall\,I \in \mathcal{I}   \label{IOM:LinDep} \\
	  &\boldsymbol{\eta}  &\in   &\{0,1\}^{k'} 
            \label{IOM:EtaBin}
    \end{LPArray}\end{NamedSubs}
where $\mathcal I$ is the collection of subsets of $[\![k]\!]$ that correspond to linearly dependent sets of rows from $(A'\vert\mathbf{b}')$. Notice that constraint \eqref{IOM:Feas} is exactly \eqref{MBLP-CR:Feas} to ensure feasibility to the continuous relaxation. Constraint \eqref{IOM:Tight} is a mirror of \eqref{IOM:Feas} with the inequality swapped and a big-M term to allow its deactivation; notice that if $\eta_i = 1$, then the $i^{\text{th}}$ inequality in \eqref{IOM:Feas} must be tight.  Finally, constraint \eqref{IOM:Extrm} ensures that the required $m + n$ constraints are tight while \eqref{IOM:LinDep} maintains that the tight constraints are linearly independent; ie. that they properly correspond to an extreme point.

\begin{proposition}\label{prop:IOMshowsIdeal}
The optimal objective value of \ref{model:IOM} is zero if and only if the MBPL embedded in constraint \eqref{IOM:Feas} is ideal.
\end{proposition}
\begin{corollary}\label{cor:IOMsubset}
The MBLP is ideal if and only if there exists a subset $\mathcal{J}$ of $\mathcal{I}$ for which the optimal objective value of \ref{model:IOM} is zero when constraint \eqref{IOM:LinDep} is implemented only for sets $I \in \mathcal{J}$.
\end{corollary}
\begin{proof}
Proposition \ref{prop:IOMshowsIdeal} follows from the construction of \ref{model:IOM}: if \ref{model:IOM} returns an optimal value of zero, then the tight-intersection of any collection of $m+n$ linearly dependent constraints of the MBLP has no fractional $y$. That is, none of the MBLP's extreme points are fractional.

Let $\mathcal{P}$ represent the feasible region of \ref{model:IOM} and, for any subset $\mathcal{J}$ of $\mathcal{I}$, let $\mathcal{P}_\mathcal{J}$ represent its feasible region when constraint \eqref{IOM:LinDep} is implemented only for sets $I \in \mathcal{J}$. Notice that $\mathcal{P} \subseteq \mathcal{P}_\mathcal{J}$ since $\mathcal{P}_\mathcal{J}$ is defined by a subset of the constraints defining $\mathcal{P}$. Naturally, then, 
    $$
    \{\Phi(\mathbf{y}):(\mathbf{x},\mathbf{y},\boldsymbol{\eta})\in\mathcal{P}\} 
    \hquad\subseteq\hquad
    \{\Phi(\mathbf{y}):(\mathbf{x},\mathbf{y},\boldsymbol{\eta})\in\mathcal{P}_\mathcal{J}\}
    $$
which means that 
    $$
    \max\{\Phi(\mathbf{y}):(\mathbf{x},\mathbf{y},\boldsymbol{\eta})\in\mathcal{P}\} 
    \LEQ
    \max\{\Phi(\mathbf{y}):(\mathbf{x},\mathbf{y},\boldsymbol{\eta})\in\mathcal{P}_\mathcal{J}\}.
    $$
for any subset $\mathcal{J}$ of $\mathcal{I}$. But $\Phi(\mathbf{y}) \geq 0$ because $\mathbf{y} \in [0,1]^n$, so it must be that $\max\{\Phi(\mathbf{y}):(\mathbf{x},\mathbf{y},\boldsymbol{\eta})\in\mathcal{P}\} = 0$ if there exists a subset $\mathcal{J}$ of $\mathcal{I}$ for which $\max\{\Phi(\mathbf{y}):(\mathbf{x},\mathbf{y},\boldsymbol{\eta})\in\mathcal{P}_\mathcal{J}\} = 0$.
\end{proof}

\subsection{Separation Problem for Dependence Covers \ref*{IOM:LinDep}}
Often, only a small percentage of the elements of $\mathcal{I}$ correspond to useful constraints. In such cases, it is more effective to implement \textit{dependence covers} \eqref{IOM:LinDep} as the need arises; that is, as solutions which are tight to a linearly dependent set of rows of $(A'\vert\mathbf{b}')$ are identified. Let $(\mathbf{x}^*, \mathbf{y}^*, \boldsymbol{\eta}^*)$ be an optimal solution to an instance of \ref{model:IOM} with an under-full dependence set $\mathcal{I}$ and let $T = \big\{i \in [\![k']\!]\ :\ \eta^*_i = 1\big\}$. Define the matrix $(A_T|\mathbf{b}_T) \in \mathbb{R}^{(m+n)\times(m+n+1)}$ to be made up of the rows of $(A'|\mathbf{b}')$ which correspond to the elements of $T$ so that  $A_T\tvect{\mathbf{x^*}}{\mathbf{y^*}} = \mathbf{b}_T$.

If $\text{rank}(A_T|\mathbf{b}_T) = |T| = m+n$, then $(\mathbf{x}^*,\mathbf{y}^*)$ is an extreme point of the MBLP and no further action is required. Suppose, on the other hand, that $A_T$ is rank-deficient. In this case, the point $(\mathbf{x}^*,\mathbf{y}^*)$ may not be extreme and the corresponding indicator vector $\boldsymbol{\eta}^*$ should be removed from the feasible region of \eqref{model:IOM}. It is sufficient to append the collection $T$ of tight constraints directly to $\mathcal{I}$ and repeat the optimization. However, $T$ may include some unnecessary indices which will reduce the strength of the corresponding cover inequality \eqref{IOM:LinDep}. A better approach is to identify a minimal linearly dependent subset (or \textit{circuit}) of $T$ by solving the following minimization problem:
    \begin{equation}\label{seSU1d:lbjibstract}
    \min\big\{\|\mathbf{p}\|_0\ :\ (A_T|\mathbf{b}_T)^\top\mathbf{p} = \mathbf{0},\ \mathbf{p} \in \mathbb{R}^{|T|}\setminus \{\mathbf{0}\}\big\}.
    \end{equation}
The optimal value of this problem is sometimes called the \textit{spark} of matrix $(A_T|\mathbf{b}_T)^\top$; indices with a nonzero optimal multiplier $p_i$ correspond to elements of the minimal circuit. As demonstrated in \cite{RestrictedIsometry2014TillmannPfetsch}, finding the spark of a matrix is generally NP-Hard, so this problem is often approximated by substituting the $\ell_0$-norm for a $\ell_1$-norm. Sometimes called basis-pursuit, this approximation can be solved as an LP \cite{AtomicDecompositionBasisPirsuit2001ChenDonohoSaunders}. There are known conditions under which the $\ell_0$ and $\ell_1$ solutions coincide \cite{SufficientConditionsSparseSignalRecovery2011JuditskyNemirovski} and a fast solution to this problem may be important for finding dependence covers on larger binary programs. However, the following MBLP implemented in Gurobi gave nice, integer solutions and converged quickly on our small binary programs. It performed well as a separation kernel.
    \begin{subequations}\label{sep:IOMg}
    \begin{LPArray}{Minimize}{\sum_{i\in T} \mu_i + \nu_i}
    & (A_T|\mathbf{b}_T)^\top\mathbf{p} = \mathbf{0} \\
    & 1-(M+1)(1-\boldsymbol{\mu})  \LEQ  \mathbf{p}  \LEQ  M\boldsymbol{\mu} \\
    & -M\boldsymbol{\nu}  \LEQ  \mathbf{p}  \LEQ  (M+1)(1-\boldsymbol{\nu})-1 \\
    & \boldsymbol{\mu},\boldsymbol{\nu}  \IN  \{0,1\}^{|T|} \\
    & \mathbf{p}  \IN  \mathbb{R}^{|T|}.
    \end{LPArray}
    \end{subequations}

 \noindent We have $\mu_i = 1$ if $p_i$ is positive and $\nu_i = 1$ if it is negative. We also ensure that active $\mathbf{p}$ multipliers take an absolute value greater than one to prevent numerical errors and promote integer multipliers. 

 \begin{remark}
We can choose $M$ based on the encoding size of the problem.  
According to \cite[Theorem 10.2]{Schrijver1998-dd}, for a rational polyhedron $P = \{x \in \mathbb{R}^n : Ax \leq b\}$, 
vertex complexity $\nu$ (the binary encoding size of any vertex or extreme ray of $P$), is bounded by $4n^2 \varphi$, where $\varphi$ is the facet complexity of $\nu$, i.e., the binary encoding size of the constraints.  

Let $\bar x$ be any solution to $Ax = 0$.  Next, $A'x \leq b'$ be a modification of $Ax = 0$ such that $x_i \leq -1$ whenever $\bar x_i < 0$, $x_i \geq 1$ whenever $\bar x_i > 0$.  Then there exists a scaling of  $\bar x$ that is valid for $A'x \leq b'$.  The facet complexity of $A'x \leq b'$ is then bounded by $(m+n)n (\log(|A_{max}|)+2)$, assuming that $|A_{max}| \geq 1$.  Thus, any vertex is bounded in infinity norm by 
$$
M= 2^{4n^2 \nu} = 2^{4n^3(m+n)\log(|A_{max}| + 2)}.
$$
\end{remark}

There also exist MILP formulations of this problem which do not require big-$M$ constants; see \cite{Tillmann2019}. We opted to use the Big-$M$ formulation \eqref{sep:IOMg} since, when applied to our small binary programs, it quickly converges to integer multipliers which where helpful in identifying and demonstrating the general dependencies given in Lemmas \ref{lem:SUCoversNecessary}, \ref{lem:RUCoversNecessary}, \ref{lem:SBMCoversNecessary1}, and \ref{lem:SBMCoversNecessary2}.

Regardless of the approach used to find it, let $\mathbf{p}^*$ be an optimal (or nearly optimal) solution to \eqref{seSU1d:lbjibstract} so that $T' = \{i\in T\ :\ p_i^*\neq0\}$ represents a minimal row-circuit of $(A_T|\mathbf{b}_T)$ which, upon being added to $\mathcal{I}$ and used in constraint \eqref{IOM:LinDep}, represents a strong valid inequality which separates $\boldsymbol{\eta}^*$.

\textbf{Remark:} Any \textit{equality} constraints in the original MBLP are necessarily tight and can be safely discarded from constraint \eqref{IOM:Tight} so long as their number is also subtracted from the right side of \eqref{IOM:Extrm}. However, it is important to re-introduce any discarded equalities into $(A_T|\mathbf{b}_T)$ during the separation problem as they can contribute to linear dependence. It is perfectly valid to split any equality constraints $D\tvect{\mathbf{x}}{\mathbf{y}} = \mathbf{d}$ into inequalities $D\tvect{\mathbf{x}}{\mathbf{y}} \leq \mathbf{d}$ and $-D\tvect{\mathbf{x}}{\mathbf{y}} \leq -\mathbf{d}$, but this might slow things down. Our implementation discards and re-introduces equality constraints.

\subsection{Parametric Implementation}
In its linear form, \ref{model:IOM} can only tell us whether the MBLP is ideal for a given set of input data. Showing more general idealness requires letting the input data ($A$ and $\mathbf{b}$) vary within the framework of \ref{model:IOM}. Such a parametric implementation is a Mixed-Binary Quadratically Constrained Program that actively searches for an instance of the MBLP that has a fractional extreme point. If the optimal objective value of this parametric implementation is zero, then the MBLP is ideal for any input data.

A parametric implementation may be tractable if a small but comprehensive subset of $\mathcal{I}$ can be identified for a given MBLP. 

\begin{remark}
A data-dependent implementation can be used to solve a separation problem for the parametric implementation if a comprehensive dependence set $\mathcal{I}$ is not known. However, any identified dependence covers should be evaluated to verify whether they are indeed dependent in general.
\end{remark}

\section{Computer-aided Proofs}\label{sec:compproofs}
In this section, we use \ref{model:IOM} to demonstrate the idealness of the various formulations or else to find non-ideal examples of them. See our GitHub Repository \hyperlink{https://github.com/jfravel/Ideal-O-Matic}{https://github.com/jfravel/Ideal-O-Matic} for the Python, model, log, and solution files supporting this result.

\subsection{\ref{model:SU} is Pairwise-Ideal (Computer-aided proof of Theorem~\ref{thm:SUideal})}\label{sec:SUideal}
Consider the continuous relaxation of \ref{model:SU}:

    \ParaTitle{Continuous Relaxation of \ref{model:SU}}{rSU}\label{model:rSU}
    \begin{NamedSubs}{rSU}\begin{FeasRegion}
    c_{ls}  \GEQ  \LB_{ls} + (\LB_{ks} + \PM_{kls} - \LB_{ls})\delta_{kls}	 
        & \forall\ (k,l,s)\in\combs_{ij} \label{rSU:lb}   \\
    c_{ks}  \LEQ  \UB_{ks} + (\UB_{ls} - \PM_{kls} - \UB_{ks})\delta_{kls}	 
        & \forall\ (k,l,s)\in\combs_{ij} \label{rSU:ub}   \\
    c_{ks} - c_{ls}  \LEQ  \UB_{ks} - \LB_{ls} + (\LB_{ls} - \PM_{kls} - \UB_{ks})\delta_{kls}   
        & \forall\ (k,l,s)\in\combs_{ij} \label{rSU:prec} \\
    \delta_{kls}  \GEQ  0   
        & \forall\ (k,l,s)\in\combs_{ij} \label{rSU:indic} \\
    \delta_{ijx} + \delta_{jix} + \delta_{ijy} + \delta_{jiy}  \EQ  1 	\label{rSU:disj}
    \end{FeasRegion}\end{NamedSubs}

\begin{lemma}\label{lem:SUCoversNecessary}
    The collection of constraints $\inst{rSU:lb}{}$, $\inst{rSU:ub}{}$, $\inst{rSU:prec}{}$, and $\inst{rSU:indic}{}$ is linearly dependent for any given index $(k,l,s)\in\combs_{ij}$.
\end{lemma}
\begin{proof}
    Assuming that \eqref{rSU:indic} is tight, then $\delta_{kls} = 0$ and the remaining constraints reduce to $c_{ls}  \geq  \LB_{ls}$, $c_{ks}  \leq  \UB_{ks}$, and $c_{ks} - c_{ls}  \leq  \UB_{ks} - \LB_{ls}$ which are linearly dependent when tight.
\end{proof}

These collections of constraints are sufficient to demonstrate the pairwise idealness of \ref{model:SU} via a parametric implementation of \ref{model:IOM}. Note that it is possible to prove minimal linear dependence on the above collections of constraints under the assumption that $\PM_{kls}  \neq  \UB_{ls} - \LB_{ks}$ for each index $(k,l,s)\in\pairs_{ij}$; this stronger lemma and its proof appear in Appendix \ref{app:StongLemmas} since they are unnecessary for the bulk of this article. 

\begin{proof}[Computational Proof of Theorem \ref{thm:SUideal}] [\underline{Up to $\epsilon = \frac{r}{10}$}]
We begin by implementing \ref{model:IOM} on \ref{model:rSU} with:
    \begin{itemize}
    \item Region size $r = 10$, a big-$M$ value of $r$, and a small constant $\epsilon = 1$;
    \item Parametric input data ($\LB$, $\UB$, and $\PM$)
    \item Under the constraint that $\PM_{kls}  \leq  \UB_{ls} - \LB_{ks} - \epsilon$ for each index $(k,l,s)\in\pairs_{ij}$; and
    \item \ref{IOM:LinDep} implemented only on the collections of constraints identified in Lemma \ref{lem:SUCoversNecessary}.
    \end{itemize}
Setting $r = 10$ is done w.l.o.g. because any non-ideal values of the parameters can be scaled. This value is also sufficiently large for the big-$M$ terms. Letting Gurobi evaluate this optimization problem with parameters \verb|NonConvex=2|, \verb|NumericFocus=3|, \verb|FeasibilityTol=1e-9|, and \verb|IntFeasTol=1e-5| gives an optimal value of \verb|6.89e-5| (which is less than $n$ times the integer-feasibility-tolerance \verb|IntFeasTol|) in about 36 seconds. This value is effectively zero, so the result holds by Corollary \ref{cor:IOMsubset}. 
\end{proof}

\subsection{\ref{model:RU} is Pairwise-Ideal (Computer-aided proof of Theorem~\ref{thm:RUideal}}\label{sec:RUideal}
Consider the continuous relaxation of \ref{model:RU}:

    \ParaTitle{Continuous Relaxation of \ref{model:RU}}{rRU}\label{model:rRU}
    \begin{NamedSubs}{rRU}\begin{FeasRegion}
    c_{ls}  \GEQ  \LB_{ls} + (\LB_{ks} + \PM_{kls} - \LB_{ls})\delta_{kls}	
        & \forall\ (k,l,s)\in\combs_{ij} 	\label{RUR:lb}   \\
    c_{ks}  \LEQ  \UB_{ks} + (\UB_{ls} - \PM_{kls} - \UB_{ks})\delta_{kls}	
        & \forall\ (k,l,s)\in\combs_{ij}  \label{RUR:ub}   \vspace{.5em}\\
    \begin{aligned}
    c_{kx} - c_{lx}  \LEQ  \PM_{lks} &- (\PM_{lks} + \PM_{kls})\delta_{kls} \\
        &\quad + (\UB_{ks} - \PM_{lks} - \LB_{ls})\delta_{lks}
    \end{aligned}
        & \forall\ (k,l,s)\in\combs_{ij} 	\label{RUR:prec} \vspace{.5em}\\
    \delta_{kls}  \GEQ  0 
        & \forall\ (k,l,s)\in\combs_{ij}  \label{RUR:indic} \\
    \delta_{ijs} + \delta_{jis}  \LEQ  1 
        & \forall\ s \in \{x,y\}	\label{RUR:tight} \\
    \delta_{ijx} + \delta_{jix} + \delta_{ijy} + \delta_{jiy}  \GEQ  1 
        \label{RUR:disj}.
    \end{FeasRegion}\end{NamedSubs}

In the following lemma and its proof, we use the term \textit{instance} in reference to a constraint as it applies to a specific index. For example, the $(i,j,x)$ instance of \eqref{RUR:lb}, denoted $\inst{RUR:lb}{ijx}$, is given by
	$$
	c_{jx}  \GEQ  \LB_{jx} + (\LB_{ix} + \PM_{ijx} - \LB_{jx})\delta_{ijx}
	$$
\begin{lemma}\label{lem:RUCoversNecessary}
    The following collections of constraints are linearly dependent for all indices $(k,l,s) \in J$ when assumed tight:
    \begin{enumerate}[label=(\roman*)]
    \item $\inst{RUR:disj}{}$, $\inst{RUR:tight}{s}$, $\inst{RUR:indic}{kls'}$, and $\inst{RUR:indic}{lks'}$ where $s'\in\{x,y\}\setminus s$;
    
    \item $\inst{RUR:lb}{kls}$, $\inst{RUR:lb}{lks}$, $\inst{RUR:prec}{kls}$, $\inst{RUR:tight}s$, and $\inst{RUR:indic}{lks}$;
    
    \item $\inst{RUR:lb}{kls}$, $\inst{RUR:ub}{kls}$, $\inst{RUR:prec}{kls}$, $\inst{RUR:tight}{s}$, and $\inst{RUR:indic}{kls}$; and
            
    \item $\inst{RUR:ub}{kls}$, $\inst{RUR:ub}{lks}$, $,\inst{RUR:prec}{kls}$, $\inst{RUR:tight}{s}$, and $\inst{RUR:indic}{lks}$.
    \end{enumerate}
\end{lemma}
\begin{proof}We demonstrate the linear dependence of each collection by assuming tightness and considering a linear combination of the constraints:
    \begin{enumerate}[label=\emph{(\roman*)}]
    \item Assuming tightness, the combination $\inst{RUR:disj}{} - \inst{RUR:tight}{s} -  \inst{RUR:indic}{kls'} - \inst{RUR:indic}{lks'}$ reduces to $0 = 0$.
    
    \item Assuming tightness on $\inst{RUR:tight}{s}$ and $\inst{RUR:indic}{lks}$ gives $\delta_{lks} = 0$ and $\delta_{kls} = 1$ and the remaining constraints reduce to $c_{ls}  \geq  \LB_{ks} + \PM_{kls}$,\ $c_{ks}  \geq  \LB_{ks}$, and $c_{ls} - c_{ks}  \geq  \PM_{kls}$ which are linearly dependent when assumed tight.
    
    \item Similarly, assuming tightness on $\inst{RUR:tight}{s}$ and $\inst{RUR:indic}{kls}$ gives $\delta_{lks} = 1$ and $\delta_{kls} = 0$ so that the remaining constraints reduce to $c_{ls}  \geq  \LB_{ls}$, $c_{ks}  \leq  \UB_{ks}$, and $c_{ks} - c_{ls}  \leq  \UB_{ks} - \LB_{ls}$ which are linearly dependent when assumed tight.
    
    \item Again, assuming tightness on $\inst{RUR:tight}{s}$ and $\inst{RUR:indic}{lks}$ gives $\delta_{lks} = 0$ and $\delta_{kls} = 1$. The remaining constraints reduce to $c_{ks}  \leq  \UB_{ls} - \PM_{kls}$, $c_{ls}  \leq  \UB_{ls}$, and $c_{ks} - c_{ls}  \leq  \PM_{lks} - \PM_{kls} - \PM_{lks}$ which are linearly dependent when assumed tight.
    \end{enumerate}
\end{proof}

Like Lemma \ref{lem:SUCoversNecessary}, the collections of constraints detailed above are minimally linearly dependent under the assumption that $\PM_{kls}  \neq  \UB_{ls} - \LB_{ks}$ for each index $(k,l,s)\in\pairs_{ij}$. The stronger version of Lemma \ref{lem:RUCoversNecessary} and its proof appear in Appendix \ref{app:StongLemmas}. It is also true that the collection $\inst{RUR:lb}{kls}$, $\inst{RUR:ub}{kls}$, $\inst{RUR:prec}{kls}$, $\inst{RUR:disj}{}$, $\inst{RUR:indic}{kls'}$, $\inst{RUR:indic}{lks'}$, and $\inst{RUR:indic}{kls}$ is minimally linearly independent. However, assuming tightness on $\inst{RUR:disj}{}$, $\inst{RUR:indic}{kls'}$, $\inst{RUR:indic}{lks'}$, and $\inst{RUR:indic}{kls}$ implies that $\delta_{kls} = 0,\ \delta_{lks} = 1,\ \delta_{kls} = 0,$ and $\delta_{lks} = 0$ which satisfies the integer requirements. There are other collections that similarly alter cases 1 and 3 in the above lemma, but none are important to \ref{model:IOM} because they necessarily imply an integer solution. The collections detailed in Lemma \ref{lem:RUCoversNecessary} are sufficient to demonstrate the idealness of \ref{model:RU} via a parametric implementation of \ref{model:IOM}.

\begin{proof}[Computational Proof of Theorem \ref{thm:RUideal}][\underline{Up to $\epsilon = \frac{r}{10}$}]
This proof parallels that of Theorem \ref{thm:SUideal}. We begin by implementing \ref{model:IOM} on \ref{model:rRU} with:
    \begin{itemize}
    \item Region size $r = 10$, a big-$M$ value of $r$, and a small constant $\epsilon = 1$;
    \item Parametric input data ($\LB$, $\UB$, and $\PM$)
    \item Under the constraint that $\PM_{kls}  \leq  \UB_{ls} - \LB_{ks} - \epsilon$ for each index $(k,l,s)\in\pairs_{ij}$; and
    \item \ref{IOM:LinDep} implemented only on the collections of constraints identified in Lemma \ref{lem:RUCoversNecessary}.
    \end{itemize}
Setting $r = 10$ is done w.l.o.g. because any non-ideal values of the parameters can be scaled. This value is also sufficiently large for the big-$M$ terms. Letting Gurobi Evaluate this optimization problem with parameters \verb|NonConvex=2|, \verb|NumericFocus=3|, \verb|FeasibilityTol=1e-9|, and \verb|IntFeasTol=1e-5| gives an optimal value of \verb|4.195e-05| (which is less than $n$ times the integer-feasibility-tolerance \verb|IntFeasTol|) in about 523 seconds. This value is effectively zero, so the Theorem holds by Corollary \ref{cor:IOMsubset}. 
\end{proof}

\subsection{\ref{model:SB-M} is Pairwise-Ideal}\label{sec:SBMideal}
Recall the theorem in question, \SBMideal*
\noindent\,and consider the continuous relaxation of \ref{model:SB-M}:
    \ParaTitle{Continuous Relaxation of \ref{model:SB-M}}{rSB}\label{model:rSB}
    \begin{NamedSubs}{rSB}\begin{FeasRegion}
    c_{ls}  \GEQ  \LB_{ks} + \PM_{kls} - (\LB_{ks} + \PM_{kls} - \LB_{ls})\, \bcf<\tilde>{\boldsymbol{\bar\delta}_{kls},\boldsymbol{\delta}_{ij}}
        & \forall\ (k,l,s)\in\combs_{ij} \label{rSB:lb}   \\
    c_{ks}  \LEQ  \UB_{ls} - \PM_{kls} - (\UB_{ls} - \PM_{kls} - \UB_{ks})\, \bcf<\tilde>{\boldsymbol{\bar\delta}_{kls},\boldsymbol{\delta}_{ij}}	 
        & \forall\ (k,l,s)\in\combs_{ij} \label{rSB:ub}   \\   
    c_{ls} - c_{ks}  \GEQ  \PM_{kls} + (\LB_{ls} - \PM_{kls} - \UB_{ks})\, \bcf<\tilde>{\boldsymbol{\bar\delta}_{kls},\boldsymbol{\delta}_{ij}}
        & \forall\ (k,l,s)\in\combs_{ij} \label{rSB:prec} \\
    \delta_{ij} + \delta_{ji} - \Delta_{ij}  \LEQ  1  \label{rSB:McCor3}\\
    \delta_{kl} - \Delta_{ij}  \GEQ  0  & \forall\ k,l\in\{i,j\}  \label{rSB:McCor12}\\
    \delta_{kl}  \GEQ  0  & \forall\ k,l\in\{i,j\}  \label{rSB:db-}\\
    \delta_{kl}  \LEQ  1  & \forall\ k,l\in\{i,j\}  \label{rSB:db+}
    \end{FeasRegion}\end{NamedSubs}

Again, we use the term \textit{instance} in reference to a constraint as it applies to a specific index. For example, the $(i,j,x)$ instance of \eqref{rSB:lb}, denoted $\inst{rSB:lb}{ijx}$, is given by
	$$
	c_{jx}  \GEQ  \LB_{ix} + \PM_{ijx} - (\LB_{ix} + \PM_{ijx} - \LB_{jx})\, \bcf<\tilde>{\boldsymbol{\bar\delta}_{ijx},\boldsymbol{\delta}_{ij}}
	$$
\begin{lemma}\label{lem:SBMCoversNecessary1} The following collections of constraints are linearly dependent in general:\hfill~
    \begin{enumerate}[label=(\roman*)]
    
    \item \eqref{rSB:McCor3}, $\inst{rSB:McCor12}{kl}$, and $\inst{rSB:db+}{lk}$ for either index $(k,l) \in \{(i,j),(j,i)\}$;\vspace{.5em}

    \item $\inst{rSB:lb}{ijx}$, $\inst{rSB:ub}{ijx}$, and $\inst{rSB:prec}{ijx}$ with:\vspace{.5em}\\
        \begin{enumerate*}[label=(\alph*)]
        \item \eqref{rSB:McCor3}, \hspace{1.5em}
        \item $\inst{rSB:McCor12}{ij}$ and $\inst{rSB:db+}{ji}$, or \hspace{1.5em}
        \item $\inst{rSB:McCor12}{ji}$ and $\inst{rSB:db+}{ij}$.
        \end{enumerate*}\vspace{.5em}

    \item $\inst{rSB:lb}{ijy}$, $\inst{rSB:ub}{ijy}$, and $\inst{rSB:prec}{ijy}$ with:\vspace{.5em}\\
        \begin{enumerate*}[label=(\alph*)]
        \item $\inst{rSB:McCor12}{ij}$ or \hspace{1.5em}
        \item \eqref{rSB:McCor3} and $\inst{rSB:db+}{ji}$.
        \end{enumerate*}\vspace{.5em}

    \item $\inst{rSB:lb}{jix}$, $\inst{rSB:ub}{jix}$, and $\inst{rSB:prec}{jix}$ with:\vspace{.5em}\\
        \begin{enumerate*}[label=(\alph*)]
        \item $\inst{rSB:McCor12}{ij}$ and $\inst{rSB:db-}{ij}$ or \hspace{1.5em}
        \item $\inst{rSB:McCor12}{ji}$ and $\inst{rSB:db-}{ji}$.
        \end{enumerate*}\vspace{.5em}
    
    \item $\inst{rSB:lb}{jiy}$, $\inst{rSB:ub}{jiy}$, and $\inst{rSB:prec}{jiy}$ with:\vspace{.5em}\\
        \begin{enumerate*}[label=(\alph*)]
        \item $\inst{rSB:McCor12}{ji}$ or \hspace{1.5em}
        \item \eqref{rSB:McCor3} and $\inst{rSB:db+}{ij}$.
        \end{enumerate*}\vspace{.5em}
    \end{enumerate}
\end{lemma}
\begin{proof} We demonstrate the linear dependence of each collection by assuming tightness and considering a linear combination of the constraints:
    \begin{enumerate}[label=(\roman*)]
    
    \item Assuming tightness, the combination $\eqref{rSB:McCor3} - \inst{rSB:McCor12}{kl} - \inst{rSB:db+}{lk}$ reduces to $0 = 0$.

    \item In each subcase, assuming tightness implies that $\bcf<\tilde>{\boldsymbol{\bar\delta}_{ijx},\boldsymbol{\delta}_{ij}}  =  \delta_{ij} + \delta_{ji} - \Delta_{ij}  =  1$ and the remaining constraints reduce to $c_{jx}  \geq  \LB_{jx}$, $c_{ix}  \leq  \UB_{ix}$, and $c_{ix} - c_{jx}  \leq  \UB_{ix} - \LB_{jx}$ which are linearly dependent.
    
    \item In either subcase, assuming tightness implies that $\bcf<\tilde>{\boldsymbol{\bar\delta}_{ijy},\boldsymbol{\delta}_{ij}}  =  1 - \delta_{ij} + \Delta_{ij}  =  1$ and the remaining constraints reduce to $c_{jy}  \geq  \LB_{jy}$, $c_{iy}  \leq  \UB_{iy}$, and $c_{iy} - c_{jy}  \leq  \UB_{iy} - \LB_{jy}$ which are linearly dependent.

    \item In either subcase, assuming tightness implies that $\bcf<\tilde>{\boldsymbol{\bar\delta}_{jix},\boldsymbol{\delta}_{ij}}  =  1 - \Delta_{ij}  =  1$ and the remaining constraints reduce to $c_{ix}  \geq  \LB_{ix}$, $c_{jx}  \leq  \UB_{jx}$, and $c_{jx} - c_{ix}  \leq  \UB_{jx} - \LB_{ix}$ which are linearly dependent.

    \item In either subcase, assuming tightness implies that $\bcf<\tilde>{\boldsymbol{\bar\delta}_{jiy},\boldsymbol{\delta}_{ij}}  =  1 - \delta_{ji} + \Delta_{ij}  =  1$ and the remaining constraints reduce to $c_{iy}  \geq  \LB_{iy}$, $c_{jy}  \leq  \UB_{jy}$, and $c_{jy} - c_{iy}  \leq  \UB_{jy} - \LB_{iy}$ which are linearly dependent.
    \end{enumerate}
\end{proof}

\begin{lemma}\label{lem:SBMCoversNecessary2} The following collections of constraints are linearly dependent under certain conditions:\hfill~
    \begin{enumerate}[label=(\roman*)]
    \item $\inst{rSB:lb}{ijx}$, $\inst{rSB:ub}{ijx}$, $\inst{rSB:prec}{ijx}$, $\inst{rSB:lb}{ijy}$, $\inst{rSB:ub}{ijy}$, $\inst{rSB:prec}{ijy}$, and $\inst{rSB:db+}{ji}$ is linearly dependent so long as $\LB_{iy}+\PM_{ijy}-\UB_{jy} \neq 0$.

    \item $\inst{rSB:lb}{ijx}$, $\inst{rSB:ub}{ijx}$, $\inst{rSB:prec}{ijx}$, $\inst{rSB:lb}{jiy}$, $\inst{rSB:ub}{jiy}$, $\inst{rSB:prec}{jiy}$, and $\inst{rSB:db+}{ij}$ is linearly dependent so long as $\LB_{jy}+\PM_{jiy}-\UB_{iy} \neq 0$.

    \item $\inst{rSB:lb}{jix}$, $\inst{rSB:ub}{jix}$, $\inst{rSB:prec}{jix}$, $\inst{rSB:lb}{ijy}$, $\inst{rSB:ub}{ijy}$, $\inst{rSB:prec}{ijy}$, and $\inst{rSB:db-}{ij}$ is linearly dependent so long as $\LB_{iy}+\PM_{ijy}-\UB_{jy} \neq 0$.

    \item $\inst{rSB:lb}{ijx}$, $\inst{rSB:ub}{ijx}$, $\inst{rSB:prec}{ijx}$, $\inst{rSB:lb}{jiy}$, $\inst{rSB:ub}{jiy}$, $\inst{rSB:prec}{jiy}$, and $\inst{rSB:db-}{ji}$ is linearly dependent so long as $\LB_{jy}+\PM_{jiy}-\UB_{iy} \neq 0$.
    \end{enumerate}
\end{lemma}
\begin{proof}\hfill~
    \begin{enumerate}[label=(\roman*)]
    \item Assuming tightness, the combination $\inst{rSB:lb}{ijx} - \inst{rSB:ub}{ijx} - \inst{rSB:prec}{ijx} + \beta\inst{rSB:lb}{ijy} - \beta\inst{rSB:ub}{ijy} - \beta\inst{rSB:prec}{ijy} - \gamma\inst{rSB:db+}{ji}$, where $\beta = \frac{\LB_{ix}+\PM_{ijx}-\UB_{jx}}{\LB_{iy}+\PM_{ijy}-\UB_{jy}}$ and $\gamma = \LB_{ix}+\PM_{ijx}-\UB_{js}$, reduces to $0=0$.
        
    \item Assuming tightness, the combination $\inst{rSB:lb}{ijx} - \inst{rSB:ub}{ijx} - \inst{rSB:prec}{ijx} + \beta\inst{rSB:lb}{jiy} - \beta\inst{rSB:ub}{jiy} - \beta\inst{rSB:prec}{jiy} - \gamma\inst{rSB:db+}{ji}$, where $\beta = \frac{\LB_{ix}+\PM_{ijx}-\UB_{jx}}{\LB_{jy}+\PM_{jiy}-\UB_{iy}}$ and $\gamma = \LB_{ix}+\PM_{ijx}-\UB_{js}$, reduces to $0=0$.
    
    \item Assuming tightness, the combination $\inst{rSB:lb}{jix} - \inst{rSB:ub}{jix} - \inst{rSB:prec}{jix} + \beta\inst{rSB:lb}{ijy} - \beta\inst{rSB:ub}{ijy} - \beta\inst{rSB:prec}{ijy} + \gamma\inst{rSB:db+}{ji}$, where $\beta = \frac{\LB_{jx}+\PM_{jix}-\UB_{ix}}{\LB_{iy}+\PM_{ijy}-\UB_{jy}}$ and $\gamma = \LB_{jx}+\PM_{jix}-\UB_{is}$, reduces to $0=0$.
    
    \item Assuming tightness, the combination $\inst{rSB:lb}{jix} - \inst{rSB:ub}{jix} - \inst{rSB:prec}{jix} + \beta\inst{rSB:lb}{jiy} - \beta\inst{rSB:ub}{jiy} - \beta\inst{rSB:prec}{jiy} + \gamma\inst{rSB:db+}{ji}$, where $\beta = \frac{\LB_{jx}+\PM_{jix}-\UB_{ix}}{\LB_{jy}+\PM_{jiy}-\UB_{iy}}$ and $\gamma = \LB_{jx}+\PM_{jix}-\UB_{is}$, reduces to $0=0$.
    \end{enumerate}
\end{proof}

\begin{proof}[Computational Proof of Theorem \ref{thm:SBMideal}][\underline{Up to $\epsilon = \frac{r}{10}$}]
We begin by implementing \ref{model:IOM} on \ref{model:rSB} with:
    \begin{itemize}
    \item Region size $r = 10$, a big-$M$ value of $r$, and a small constant $\epsilon = 1$;
    \item Parametric input data ($\LB$, $\UB$, and $\PM$)
    \item Under the constraint that $\PM_{kls}  \leq  \UB_{ls} - \LB_{ks} - \epsilon$ for each index $(k,l,s)\in\pairs_{ij}$; and
    \item \ref{IOM:LinDep} implemented only on the collections of constraints identified in Lemmas \ref{lem:SBMCoversNecessary1} and \ref{lem:SBMCoversNecessary2}.
    \end{itemize}
Setting $r = 10$ is done w.l.o.g. because any non-ideal values of the parameters can be scaled. This value is also sufficiently large for the big-$M$ terms. Letting Gurobi Evaluate this optimization problem with parameters \verb|NonConvex=2|, \verb|NumericFocus=3|, \verb|FeasibilityTol=1e-8|, and \verb|IntFeasTol=1e-5| gives an optimal value of \verb|1.494e-07| (which is less than $n$ times the integer-feasibility-tolerance \verb|IntFeasTol|) in about 26 seconds. This value is effectively zero, so the Theorem holds by Corollary \ref{cor:IOMsubset}. 
\end{proof}

\section{Conclusions}\label{sec:Conclusions}
We present our final computation comparison in Table \ref{tab:Averages-15} by averaging the results of four different 15 object instances. We find that the \ref{model:RU} formulation, with its innate tendency towards strong spacial branching, performs well without additional branching priorities but looses much of its advantage when they are present. The information-dense \ref{model:SB-L} and \ref{model:SB-M} formulations better capitalize on the symmetry breaking inequalities and branching rules. The benefit of turning off Gurobi's Cuts and Heuristics parameters remains clear. 

\begin{table}[h!]\centering
    \caption{The average runtime and percent optimally gap over four instances of 15 objects. Gurobi's Cuts and Heuristics were active in the first row but turned off in the second. The second column has both sequence pair inequalities and branching priority while the first has neither.}
    \label{tab:Averages-15}
    \setlength{\tabcolsep}{4pt}
    \begin{tabular}{|ccc|ccc||ccc|}
    \cline{4-9}    \multicolumn{2}{r}{} & \multicolumn{1}{r|}{} & \multicolumn{3}{c||}{Default Settings} & \multicolumn{3}{c|}{Sequence Pair and Branching} \\
    \hline
    C & H & Model & Runtime & MIP Gap & Nodes & Runtime & MIP Gap & Nodes \\
    \multirow{4}[1]{*}{\checkmark} & \multirow{4}[1]{*}{\checkmark} & SU    & 1842s & 1.71\% & 2.E+05 & 1894s & 2.04\% & 8.E+04 \\
          && RU    & 274s  & 0\%   & 6.E+04 & 1802s & 2.04\% & 6.E+04 \\
          && SB-L   & 1862s & 1.79\% & 3.E+05 & 545s & 0\%   & 4.E+04 \\
          && SB-M   & 2700s & 4.46\% & 5.E+05 & 508s & 0\%   & 5.E+04 \\
    \hline
    C & H & Model & Runtime & MIP Gap & Nodes & Runtime & MIP Gap & Nodes \\
    \multirow{4}[1]{*}{\xmark} & \multirow{4}[1]{*}{\xmark} & SU    & 1903s & 0.60\% & 4.E+06 & 1359s & 2.04\% & 3.E+05 \\
          && RU    & 281s & 0\%   & 5.E+05 & 1301s & 2.04\% & 3.E+05 \\
          && SB-L   & 944s & 0\%   & 2.E+06 & 93s & 0\%   & 7.E+04 \\
          && SB-M   & 2700s & 4.45\% & 2.E+06 & 119s & 0\%   & 6.E+04 \\
    \hline
    \end{tabular}%

    \end{table}

We have developed a framework for generating computer-aided proofs of idealness for mixed binary linear programs and used it to confirm a prominent conjecture made by the authors of \cite{StrongFloorLayout2017HuchetteVielma} and correct a non-ideal formulation for the rectangle packing problem without significantly increasing its size. We also demonstrated the performance of this and other formulations on the rectangular strip-packing problem with clearances.

In future work, we would like to  connect our Ideal O'Matic  to work with solvers that perform rational arithmetic like \hyperlink{https://www.scipopt.org/}{SCIP} and then use \hyperlink{https://github.com/ambros-gleixner/VIPR}{VIPR} \cite{VerifyingIntegerProgramming2017CheungGleixnerSteffy} to output certificates that serve as proofs of idealness. There also exists a variety of objective functions which could be applied to these formulations that may behave differently; in particular, an objective that is not so far removed from the binary variables may perform better under Gurobi's default branch-and-cut scheme.

\printbibliography

\newpage
\begin{appendices}
\section{Additional Tables and Plots}\label{sec:AdditionalTablesAndPlots}

    \begin{table}[ht]\centering
    \caption{Computational results for the \ref{model:SU}, \ref{model:RU}, and \ref{model:SB} formulations when turning off cuts or heuristics. Columns C and H indicate if cuts and heuristics, respectively, were active (\checkmark) or disabled (\xmark).}
    \label{tab:CutsTest}

\begin{tabular}{|ccc||cc||cc||cc||cc||}
\cline{4-11}\multicolumn{1}{c}{} &       &       & \multicolumn{2}{c||}{SU} & \multicolumn{2}{c||}{RU} & \multicolumn{2}{c||}{SB-L} & \multicolumn{2}{c||}{SB-M} \\
\hline
N     & C     & H     & Time/Gap & Nodes & Time/Gap & Nodes & Time/Gap & Nodes & Time/Gap & Nodes \\
\hline
\multirow{4}[2]{*}{10} & \cmark & \cmark & 7.79s & 1E+04 & 5.3s  & 6E+03 & 4.93s & 1E+04 & 3.95s & 7E+03 \\
      & \cmark & \xmark & 0.37s & 7E+03 & 0.53s & 9E+03 & 0.69s & 2E+04 & 0.93s & 1E+04 \\
      & \xmark & \cmark & 1.91s & 9E+03 & 4.2s  & 1E+04 & 1.76s & 1E+04 & 2.27s & 8E+03 \\
      & \xmark & \xmark & 0.22s & 7E+03 & 0.25s & 7E+03 & 0.27s & 1E+04 & 0.33s & 1E+04 \\
\hline
\multirow{4}[2]{*}{15} & \cmark & \cmark & 10.2\% & 1E+06 & 4.1\% & 1E+06 & 12\%  & 8E+05 & 18\%  & 1E+06 \\
      & \cmark & \xmark & 10.2\% & 1E+06 & 12\%  & 1E+06 & 16\%  & 1E+06 & 16\%  & 3E+06 \\
      & \xmark & \cmark & 6.1\% & 3E+06 & 10621.25s & 2E+06 & 8\%   & 3E+06 & 12\%  & 3E+06 \\
      & \xmark & \xmark & 6.1\% & 1E+07 & 7293.79s & 7E+06 & 8\%   & 2E+07 & 14\%  & 7E+06 \\
\hline
\multirow{4}[2]{*}{20} & \cmark & \cmark & 46.2\% & 2E+05 & 38.5\% & 3E+05 & 42.5\% & 4E+05 & 53.7\% & 5E+05 \\
      & \cmark & \xmark & 39.2\% & 2E+06 & 41.8\% & 6E+05 & 43.9\% & 1E+06 & 54.5\% & 1E+06 \\
      & \xmark & \cmark & 38.5\% & 9E+05 & 38.5\% & 5E+05 & 40.7\% & 8E+05 & 54.9\% & 6E+05 \\
      & \xmark & \xmark & 35.1\% & 8E+06 & 35.1\% & 5E+06 & 39.2\% & 4E+06 & 50.6\% & 3E+06 \\
\hline
\multirow{4}[2]{*}{25} & \cmark & \cmark & 54.5\% & 3E+05 & 55.1\% & 3E+05 & 53.9\% & 2E+05 & 55.1\% & 4E+05 \\
      & \cmark & \xmark & 51.9\% & 1E+06 & 55.7\% & 8E+05 & 51.9\% & 1E+06 & 56.3\% & 1E+06 \\
      & \xmark & \cmark & 53.9\% & 9E+05 & 50.6\% & 5E+05 & 54.5\% & 6E+05 & 55.1\% & 6E+05 \\
      & \xmark & \xmark & 42.9\% & 9E+06 & 43.6\% & 4E+06 & 50\%  & 4E+06 & 55.7\% & 2E+06 \\
\hline
\end{tabular}%

    \end{table}
        
    \begin{table}[ht]\centering
    \caption{Computational results for each variation in the \ref{model:SU}, \ref{model:RU}, and \ref{model:SB} formulations in terms of convergence and nodes visited. Runtime is reported if the program converged in under three hours. ``Static'' indicates that the first two constraints of each formulation have been replaced with the static bounds $c_{is} \in [\LB_{is},\UB_{is}]$; ``Branch'' indicates that the branching priorities \eqref{eq:BranchPriorityUnary} and \eqref{eq:BranchPriorityBinary} have been applied to the binary variable; and ``Sequence'' indicates that the Sequence-Pair inequalities \ref{cuts:SPU} and\ref{cuts:SPB} have been added to the formulation.}
    \label{tab:SBSPTest}
    \setlength{\tabcolsep}{4pt}
\begin{tabular}{|cc||cc||cc||cc||cc||}
\cline{3-10}\multicolumn{1}{c}{} &       & \multicolumn{2}{c||}{SU} & \multicolumn{2}{c||}{RU} & \multicolumn{2}{c||}{SB-L} & \multicolumn{2}{c||}{SB-M} \\
\multicolumn{1}{c}{} &       & Time/Gap & Nodes & Time/Gap & Nodes & Time/Gap & Nodes & Time/Gap & Nodes \\
\hline
\multirow{5}[2]{*}{10} & Standard & 0.22s & 7E+3  & 0.25s & 7E+3  & 0.27s & 1E+4  & 0.33s & 1E+4 \\
      & Static & 0.22s & 9E+3  & 2.8s  & 7E+3  & 0.16s & 1E+4  & 0.3s  & 1E+4 \\
      & Branch & 0.45s & 2E+4  & 0.31s & 1E+4  & 0.23s & 2E+4  & 0.32s & 1E+4 \\
      & Sequence & 0.51s & 1E+4  & 0.52s & 9E+3  & 0.39s & 1E+4  & 0.57s & 1E+4 \\
      & Seq \& Br & 0.31s & 5E+3  & 0.33s & 6E+3  & 0.35s & 1E+4  & 0.3s  & 1E+4 \\
\hline
\multirow{5}[2]{*}{15} & Standard & 6.1\% & 1E+7  & 7293.79s & 7E+6  & 8\%   & 2E+7  & 14\%  & 7E+6 \\
      & Static & 6\%   & 1E+8  & 6014.87s & 9E+6  & 4.1\% & 2E+8  & 10\%  & 2E+7 \\
      & Branch & 16\%  & 2E+7  & 8\%   & 1E+7  & 6\%   & 4E+7  & 12\%  & 1E+7 \\
      & Sequence & 1310.37s & 3E+6  & 3116.79s & 4E+6  & 11.52s & 5E+4  & 17.69s & 5E+4 \\
      & Seq \& Br & 6.1\% & 2E+7  & 8\%   & 1E+7  & 5.76s & 3E+4  & 7.39s & 3E+4 \\
\hline
\multirow{5}[2]{*}{20} & Standard & 35.1\% & 8E+6  & 35.1\% & 5E+6  & 39.2\% & 4E+6  & 50.6\% & 3E+6 \\
      & Static & 41.3\% & 1E+7  & 42.9\% & 3E+6  & 40\%  & 1E+7  & 47\%  & 1E+7 \\
      & Branch & 37.7\% & 8E+6  & 34.6\% & 5E+6  & 32.9\% & 8E+6  & 48.1\% & 5E+6 \\
      & Sequence & 38.5\% & 8E+5  & 41.3\% & 7E+5  & 27.7\% & 3E+6  & 31.8\% & 2E+6 \\
      & Seq \& Br & 35.1\% & 9E+5  & 35.1\% & 9E+5  & 26.2\% & 3E+6  & 28.2\% & 1E+6 \\
\hline
\multirow{5}[2]{*}{25} & Standard & 42.9\% & 9E+6  & 43.6\% & 4E+6  & 50\%  & 4E+6  & 55.7\% & 2E+6 \\
      & Static & 53.2\% & 1E+7  & 53.8\% & 3E+6  & 47.4\% & 1E+7  & 56.3\% & 9E+6 \\
      & Branch & 44.2\% & 8E+6  & 43.6\% & 5E+6  & 40\%  & 7E+6  & 55.7\% & 4E+6 \\
      & Sequence & 48.7\% & 7E+5  & 51.3\% & 7E+5  & 37\%  & 2E+6  & 42.7\% & 1E+6 \\
      & Seq \& Br & 46.1\% & 9E+5  & 47.4\% & 8E+5  & 36.3\% & 2E+6  & 36.3\% & 1E+6 \\
\hline
\end{tabular}%

    \end{table}

\begin{figure}[ht]\centering
    \begin{subfigure}[b]{0.7\textwidth}\centering
        \includegraphics[width=\textwidth]{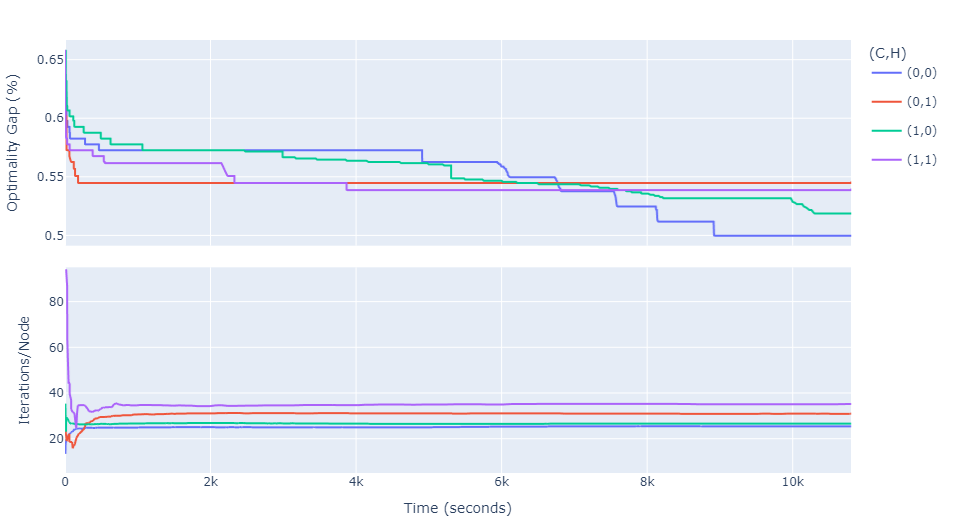}
        \caption{Cuts and Heuristics test for \ref{model:SB-L} Strip Packing with $\numobjs = 25$.}
        \label{fig:CutsTestPlot:SBL25}
    \end{subfigure}
    \end{figure}
    \begin{figure}[ht]\ContinuedFloat\centering
    \begin{subfigure}[b]{0.7\textwidth}\centering
        \includegraphics[width=\textwidth]{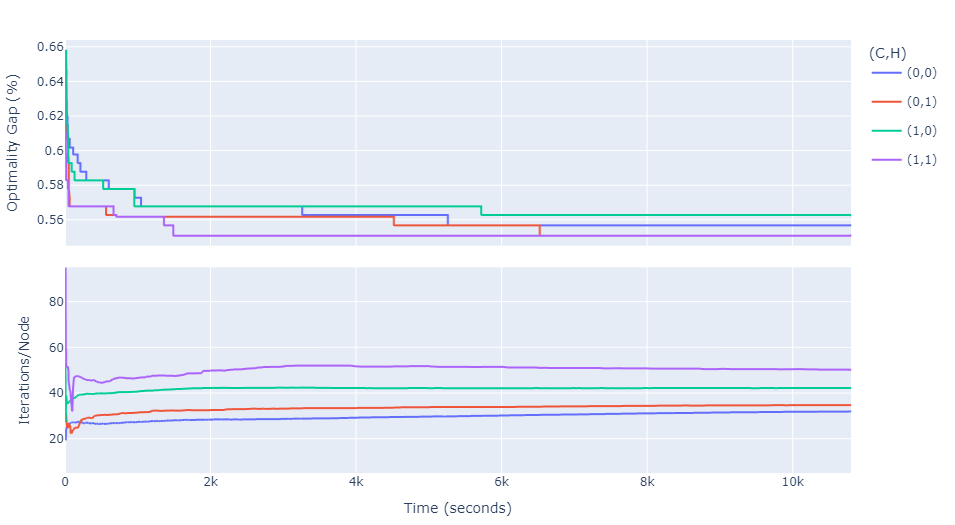}
        \caption{Cuts and Heuristics test for \ref{model:SB-M} Strip Packing with $\numobjs = 25$.}
        \label{fig:CutsTestPlot:RBM25}
    \end{subfigure}
    \caption{Plots comparing each of the $\numobjs = 25$ runs for \ref{model:SB-L} and \ref{model:SB-M} in the Cuts and Heuristics experiments (see Table \ref{tab:CutsTest}). We experiment with Gurobi's Cuts and Heuristics Parameters: Cuts are active where $C = 1$ and Heuristics are active where $H = 1$. Gaps and Iterations per Node are extracted from Gurobi's output log files via the \protect\hyperlink{https://github.com/Gurobi/gurobi-logtools}{gurobi-logtools} package for python.}
    \label{fig:CutsTestPlotsSB}
    \end{figure}

\vfill
\newpage
\section{An Analytic Proof of Theorem \ref{thm:SUideal}}\label{sec:analyticSU}
This proof is largely a copy of the proof in Appendix B of \citet{StrongFloorLayout2017HuchetteVielma}. The original proof is for the Floor Layout Problem which is a related but slightly more complicated version of rectangle packing.

\begin{proof}[Proof of Theorem \ref{thm:SUideal}]
Notice that constraints \eqref{SU:prec}, \eqref{SU:disj}, and \eqref{SU:indic} clearly enforce the non-overlapping condition. Thus we begin by demonstrating the validity of \eqref{SU:lb} and \eqref{SU:ub} by considering the following cases for each $s\in\{x,y\}$ and $\{k,l\}=\{i,j\}$:
	\begin{itemize}
	\item \underline{If $\delta_{kls} = 0$ and $\delta_{lks} = 0$}, then \eqref{SU:lb} and \eqref{SU:ub} reduce to $\LB_k \leq c_k \leq \UB_k$ which ensures that object $i$ is entirely contained within the region.
	\item \underline{If $\delta_{kls} = 1$ and $\delta_{lks} = 0$}, then \eqref{SU:lb} and \eqref{SU:ub} reduce to $\LB_k \leq c_k \leq \UB_l - \PM_{kls}$ which tightens the upper under the assumption that $(k\preced{s}l)$.
	\item \underline{If $\delta_{kls} = 0$ and $\delta_{lks} = 1$}, then \eqref{SU:lb} and \eqref{SU:ub} reduce to $\LB_l + \PM_{lkx} \leq c_k \leq \UB_l$ which tightens the lower under the assumption that $(l\preced{s}k)$.
	\item \underline{$\delta_{kls} = 1$ and $\delta_{lks} = 1$} is not feasible by \eqref{SU:disj}.
	\end{itemize}
To demonstrate the idealness of formulation \ref{model:SU}, we need to show that every extreme point solution to its relaxation has $\boldsymbol{\delta}_{ij} \in \{0,1\}^4$; or instead, by contraposition, that no solution with fractional elements in $\boldsymbol{\delta}_{ij}$ is an extreme point.

Since the $x$ and $y$ dimensions are largely independent of one another (outside of constraint \eqref{SU:disj}), we consider the following projection of the continuous relaxation of \ref{model:SU} onto just one dimension:
	\begin{subequations}\label{P}
	\begin{gather}
	c_i  \GEQ  \LB_i(1-\delta_{ji}) + (\LB_j + \PM_{ji})\delta_{ji} 	\label{P:a}
	\\
	c_i  \LEQ  \UB_i(1-\delta_{ij}) + (\UB_j - \PM_{ij})\delta_{ij} 	\label{P:b}
	\\
	c_j  \GEQ  \LB_j(1-\delta_{ij}) + (\LB_i + \PM_{ij})\delta_{ij} 	\label{P:c}
	\\
	c_j  \LEQ  \UB_j(1-\delta_{ji}) + (\UB_i - \PM_{ji})\delta_{ji} 	\label{P:d}
	\\
	c_i + \PM_{ij}  \LEQ  c_j + \BM_{ij}(1-\delta_{ij})		 		    \label{P:e}
	\\
	c_j + \PM_{ji}  \LEQ  c_i + \BM_{ji}(1-\delta_{ji})		     		\label{P:f}
	\\
	\delta_{ij} + \delta_{ji}  \LEQ  1 									\label{P:g}
	\\
	\delta_{ij}  \GEQ  0 										\label{P:h}
	\\
	\delta_{ji}  \GEQ  0 										\label{P:i}
	\end{gather}
	\end{subequations}
Now take some arbitrary solution $\big(\mathbf{\hat{c}},\boldsymbol{\hat\delta}\big)$ feasible to \eqref{P} such that $\hat\delta_{ij} \in (0,1)$. Knowing that the argument for fractional $\hat\delta_{ji}$ follows analogously, we proceed to show that $\big(\mathbf{\hat{c}},\boldsymbol{\hat\delta}\big)$ is not an extreme point of \eqref{P} by case analysis. 

Since $\hat\delta_{ij} > 0$, \eqref{P:g} implies that $\hat\delta_{ji} < 1$. Further, constraint \eqref{P:h} is not tight and our first case analysis will branch first on the tightness of constraint \eqref{P:i}. Also notice that if \eqref{P:b} and \eqref{P:c} are both tight, we would have
	\begin{gather*}
	c_j  \EQ  \LB_j + (\LB_i + \PM_{ij} - \LB_j)\hat\delta_{ij}     \\
	c_j - \LB_j + (\LB_j - \PM_{ij} - \LB_i)\hat\delta_{ij}  \EQ  0 
	\end{gather*}
	\begin{align*}
	c_i  &\EQ  \UB_i + (\UB_j - \PM_{ij} - \UB_i)\hat\delta_{ij}  \\
	c_i + \PM_{ij}  &\EQ  \UB_i + (\UB_j - \PM_{ij} - \UB_i)\hat\delta_{ij} + \PM_{ij}  \\
					&\EQ  \underline{c_j - \LB_j + (\LB_j - \PM_{ij} - \LB_i)\hat\delta_{ij}} + \UB_i + \PM_{ij} + (\UB_j - \PM_{ij} - \UB_i)\hat\delta_{ij}  \\
		            &\EQ  c_j + \UB_i + \PM_{ij} - \LB_j + (\UB_j - \PM_{ij} - \UB_i + \LB_j - \PM_{ij} - \LB_i)\hat\delta_{ij}  \\
  		            &\EQ  c_j + \UB_i + \PM_{ij} - \LB_j - (\UB_i + \PM_{ij} - \LB_j - \UB_j + \PM_{ij} + \LB_i)\hat\delta_{ij}  \\
					&\EQ  c_j + \BM_{ij} - (\BM_{ij} - \UB_j + \PM_{ij} + \LB_i)\hat\delta_{ij}  \\
					&\EQ  c_j + \BM_{ij}\big(1-\hat\delta_{ij}\big) - \big(\PM_{ij} - (\UB_j - \LB_i)\big)\hat\delta_{ij}  \\
					&\GS  c_j + \BM_{ij}\big(1-\hat\delta_{ij}\big) \qquad\text{since}\qquad \PM_{ij} < \UB_j - \LB_i
	\end{align*}
which contradicts constraint \eqref{P:e}. Thus, constraints \eqref{P:b} and \eqref{P:c} cannot both be tight. By a similar argument, under the assumption that \eqref{P:i} is not tight ($\hat\delta_{ji} > 0$), constraints \eqref{P:a} and \eqref{P:d} cannot both be tight.

	\begin{enumerate}[label=\Roman*.]
	\item Constraint \eqref{P:i} is tight.
		\begin{enumerate}[label=\roman*.]
		\item Constraint \eqref{P:f} is tight.
		\item Constraint \eqref{P:f} is \textit{not} tight.
		\end{enumerate}
	\item Constraint \eqref{P:i} is \textit{not} tight.
		\begin{enumerate}[label=\roman*.]
		\item Constraints \eqref{P:e} and \eqref{P:f} are both tight.
			\begin{enumerate}[label=\alph*.]
			\item Constraint \eqref{P:g} is tight.
			\item Constraint \eqref{P:g} is \textit{not} tight.
			\end{enumerate}
		\item At most \textit{one} of constraints \eqref{P:e} and \eqref{P:f} is tight.
			\begin{enumerate}[label=\alph*.]
			\item Constraint \eqref{P:g} is \textit{not} tight.
			\item Constraint \eqref{P:g} is tight.
			\end{enumerate}
		\end{enumerate}
	\end{enumerate}

\noindent We will show that at most three linearly independent constraints from \eqref{P} can be tight in each case.

\begin{itemize}
\item[\underline{I.i.}] Constraint \eqref{P:i} being tight implies that \eqref{P:g} is not tight. Constraint \eqref{P:f} also being tight further implies
	\begin{align}
	c_j + \PM_{ji}  &\EQ  c_i + \BM_{ji}  \nonumber\\
		&\EQ  c_i + \UB_j + \PM_{ji} - \LB_i \nonumber\\
		&\quad\Rightarrow\hquad c_j - c_i  \EQ  \UB_j - \LB_i \label{I.i.Star}.
	\end{align}
Notice that \eqref{P:a} and \eqref{P:d} reduce to $c_i  \GEQ  \LB_i$ and $c_j  \LEQ  \UB_j$ which must be tight since their difference is exactly \eqref{I.i.Star}. However, this also implies that at most three of \eqref{P:a}, \eqref{P:d}, \eqref{P:f}, and \eqref{P:i} form a linearly independent set. It remains only to show that none of the remaining constraints can be tight.

Now assume, for the sake of contradiction, that \eqref{P:b} and \eqref{P:c} are tight. Recalling that $\PM_{ij} < \UB_j - \LB_i$ and reducing, we are  respectively left with:

	\vspace{-1em}
	\begin{minipage}{0.45\textwidth}
	\begin{align*}
	\LB_i  &\EQ  \UB_i(1-\delta_{ij}) + (\UB_j - \PM_{ij})\delta_{ij} 	\\
  		   &\GS  \UB_i(1-\delta_{ij}) + \LB_i\delta_{ij}  \\
  		   &\quad\Rightarrow\hquad  \LB_i(1-\delta_{ij})  \GS  \UB_i(1-\delta_{ij})	\\
   		   &\quad\Rightarrow\hquad  \LB_i  \GS  \UB_i,
	\end{align*}
	\end{minipage}
	~\hfill~
	\begin{minipage}{0.45\textwidth}
	\begin{align*}
	\UB_j  &\EQ  \LB_j(1-\delta_{ij}) + (\LB_i + \PM_{ij})\delta_{ij}	\\
  		   &\LS  \LB_j(1-\delta_{ij}) + \UB_j\delta_{ij}  \\
  		   &\quad\Rightarrow\hquad  \UB_j(1-\delta_{ij})  \LS  \LB_j(1-\delta_{ij})	\\
   		   &\quad\Rightarrow\hquad  \UB_j  \LS  \LB_j.
	\end{align*}
	\end{minipage}\vspace{1em}
Such bounds are clearly infeasible and we have our contradiction.

The only remaining constraint is \eqref{P:e} which reduces to
	$
	\PM_{ij} - (\UB_j - \LB_i)  \leq  \BM_{ij}\big(1-\hat\delta_{ij}\big).
	$
Notice that the left-hand side is negative while the right-hand side is strictly positive, thus \eqref{P:e} cannot be tight.

\item[\underline{I.ii.}] As above, constraint \eqref{P:i} being tight implies that \eqref{P:g} is not tight. However, the non-tightness of \eqref{P:f} implies:
	\begin{align}
	c_j + \PM_{ji}  &\LS  c_i + \BM_{ji}  \nonumber\\
		&\LS  c_i + \UB_j + \PM_{ji} - \LB_i \nonumber\\
		&\hquad\Rightarrow\hquad c_j - c_i  \LS  \UB_j - \LB_i \label{I.ii.Star}.
	\end{align}
Notice that \eqref{P:a} and \eqref{P:d} reduce to $c_i  \GEQ  \LB_i$ and $c_j  \LEQ  \UB_j$; by comparing their difference to \eqref{I.ii.Star}, we find that at most one of \eqref{P:a} and \eqref{P:d} may be tight.

Assume, for the sake of contradiction, that \eqref{P:a} and \eqref{P:b} are both tight and consider their difference: 
	\begin{align*}
	c_i - c_i  &\EQ  \UB_i(1-\delta_{ij}) + (\UB_j - \PM_{ij})\delta_{ij} - \LB_i \\
	   		   &\EQ  \UB_i - \LB_i + (\UB_j - \PM_{ij} - \UB_i)\delta_{ij} \\
	   		   &\GS  \UB_i - \LB_i - (\UB_i - \LB_i)\delta_{ij} \\
	   		   &\GS  \UB_i - \LB_i \\
	   		   &\GS  0.
	\end{align*}
Thus, \eqref{P:a} and \eqref{P:b} cannot both be tight.

Assume, again for the sake of contradiction, that \eqref{P:a}, \eqref{P:c}, and \eqref{P:e} are all tight and consider the sum $\eqref{P:a} - \eqref{P:c} - \eqref{P:e}$:

	\begin{align*}
	c_i - c_j - c_i - \PM_{ij}  &\EQ  \LB_i - \LB_j(1-\delta_{ij}) - (\LB_i + \PM_{ij})\delta_{ij} - c_j  - \BM_{ij}(1-\delta_{ij})  \\
	0  &\EQ  \LB_i + \PM_{ij} - \LB_j(1-\delta_{ij}) - (\LB_i + \PM_{ij})\delta_{ij}  - \BM_{ij}(1-\delta_{ij})  \\
	0  &\EQ  (\LB_i + \PM_{ij})(1 - \delta_{ij}) - \LB_j(1-\delta_{ij}) - \BM_{ij}(1-\delta_{ij})  \\
	0  &\EQ  (\LB_i + \PM_{ij} - \LB_j - \UB_i - \PM_{ij} + \LB_j)(1-\delta_{ij})  \\
	0  &\EQ  (\LB_i - \UB_i)(1-\delta_{ij}) \hquad<\hquad 0.
	\end{align*}
This is a contradiction, so it cannot be that \eqref{P:a}, \eqref{P:c}, and \eqref{P:e} are all tight. By a similar argument, we can show that \eqref{P:b}, \eqref{P:d}, and \eqref{P:e} are all tight. Recalling that we previously demonstrated that \eqref{P:c} and \eqref{P:d} cannot both be tight. Therefore, at most two of $(\ref{P:a}-\ref{P:e})$ can be simultaneously tight in addition to \eqref{P:i} and this case has no extreme point.

\item[\underline{II.i.a.}] Constraint \eqref{P:g} being tight while \eqref{P:i} is not tight implies that \eqref{P:h} cannot be tight. 

Consider first a reduction of \eqref{P:e}, \eqref{P:f}, and \eqref{P:g}:
	\begin{gather*}
	\begin{cases}
	c_i + \PM_{ij}  \EQ  c_j + \BM_{ij}(1-\hat\delta_{ij})	\\
	c_j + \PM_{ji}  \EQ  c_i + \BM_{ji}(1-\hat\delta_{ji})	\\
	\hat\delta_{ij} + \hat\delta_{ji}  \EQ  1
	\end{cases}
	\\
	\begin{cases}
	c_i - c_j  \EQ  \BM_{ij}(1-\hat\delta_{ij}) - \PM_{ij}	\\
	c_i - c_j  \EQ  \PM_{ji} - \BM_{ji}\hat\delta_{ij}
	\end{cases}
	\\
	\PM_{ji} - \BM_{ji}\hat\delta_{ij}  \EQ  \BM_{ij}(1-\hat\delta_{ij}) - \PM_{ij}
	\\
	(\BM_{ij} - \BM_{ji})\hat\delta_{ij}  \EQ  \BM_{ij} - \PM_{ij} - \PM_{ji}
	\\
	(\BM_{ji} - \BM_{ij})\hat\delta_{ij}  \EQ  \PM_{ij} + \PM_{ji} - \UB_i - \PM_{ij} + \LB_j
	\\
	\hat\delta_{ij}  \EQ  \frac{\PM_{ji} - (\UB_i - \LB_j)}{\BM_{ji} - \BM_{ij}}.
	\end{gather*}	
Which is only positive if $\BM_{ji} < \BM_{ij}$	since the numerator is negative. Now consider expanding the denominator under this assumption:
	\begin{align*}
	\hat\delta_{ij}  &\EQ  \frac{\PM_{ji} - (\UB_i - \LB_j)}{\BM_{ji} - \BM_{ij}}
	\\
		&\EQ  \frac{\PM_{ji} - \UB_i + \LB_j}{\UB_j + \PM_{ji} -\LB_i - \UB_i - \PM_{ij} + \LB_j}
	\\
		&\EQ  \frac{(\PM_{ji} - \UB_i + \LB_j)}{(\PM_{ji} - \UB_i + \LB_j) - (\PM_{ij} - \UB_j + \LB_i)}.
	\end{align*}
Each of the parenthetical terms is negative and the denominator has a smaller absolute value than the numerator, so $\hat\delta_{ij} > 1$ which violates \eqref{P:g}. Thus this case is actually infeasible.

\item[\underline{II.i.b.}] In this case, constraints \eqref{P:e} and \eqref{P:f} are tight while \eqref{P:g} and \eqref{P:i} are not. Consider an instance where \eqref{P:a} and \eqref{P:b} are also both tight:
	\begin{gather*}
	\begin{cases}
	c_i  \EQ  \LB_i(1-\delta_{ji}) + (\LB_j + \PM_{ji})\delta_{ji}
	\\
	c_i  \EQ  \UB_i(1-\delta_{ij}) + (\UB_j - \PM_{ij})\delta_{ij}
	\\
	c_i + \PM_{ij}  \EQ  c_j + \BM_{ij}(1-\delta_{ij})
	\\
	c_j + \PM_{ji}  \EQ  c_i + \BM_{ji}(1-\delta_{ji})
	\end{cases}
	\\
	\begin{cases}
	c_i  \EQ  \LB_i + (\LB_j + \PM_{ji} - \LB_i)\delta_{ji}
	\\
	c_i  \EQ  \UB_i + (\UB_j - \PM_{ij} - \UB_i)\delta_{ij}
	\\
	c_j  \EQ  c_i + \PM_{ij} - \BM_{ij} + \BM_{ij}\delta_{ij}
	\\
	c_j  \EQ  c_i - \PM_{ji} + \BM_{ji} - \BM_{ji}\delta_{ji}
	\end{cases}
	\\
	\begin{cases}
	c_j  \EQ  \UB_i + \PM_{ij} - \UB_i - \PM_{ij} + \LB_j + (\UB_j - \PM_{ij} - \UB_i)\delta_{ij} + \BM_{ij}\delta_{ij}
	\\
	c_j  \EQ  \LB_i - \PM_{ji} + \UB_j + \PM_{ji} - \LB_i + (\LB_j + \PM_{ji} - \LB_i)\delta_{ji} - \BM_{ji}\delta_{ji}
	\end{cases}
	\\
	\begin{cases}
	c_j  \EQ  \LB_j + (\UB_j - \PM_{ij} - \UB_i + \UB_i + \PM_{ij} - \LB_j)\delta_{ij}
	\\
	c_j  \EQ  \UB_j + (\LB_j + \PM_{ji} - \LB_i - \UB_j - \PM_{ji} + \LB_i)\delta_{ji}
	\end{cases}
	\\
	\begin{cases}
	c_j  \EQ  \LB_j + (\UB_j - \LB_j)\delta_{ij}
	\\
	c_j  \EQ  \UB_j + (\LB_j - \UB_j)\delta_{ji}
	\end{cases}
	\\
	\begin{cases}
	\delta_{ij}  \EQ  \frac{c_j - \LB_j}{\UB_j - \LB_j}
	\\
	\delta_{ji}  \EQ  \frac{\UB_j - c_j}{\UB_j - \LB_j}
	\end{cases}
	\end{gather*}
	\begin{align*}
	\delta_{ij} + \delta_{ji}  &\EQ  \frac{c_j - \LB_j}{\UB_j - \LB_j} + \frac{\UB_j - c_j}{\UB_j - \LB_j} \\
		&\EQ  \frac{c_j - \LB_j + \UB_j - c_j}{\UB_j - \LB_j} \\
		&\EQ  \frac{\UB_j - \LB_j}{\UB_j - \LB_j} \\
		&\EQ  1
	\end{align*}
which violates the non-tightness assumption on \eqref{P:g}, so at most one of \eqref{P:a} and \eqref{P:b} can be tight. Similar algebra will demonstrate that at most one of \eqref{P:c} and \eqref{P:d} can be tight; so, by considering earlier arguments, we see that at most one of $(\ref{P:a}-\ref{P:d})$ can be tight. This leaves at most three tight constraints and this case has no extreme point.

\item[\underline{II.ii.a}] In this case, w.l.o.g, assume that constraint \eqref{P:e} is tight while \eqref{P:f}, \eqref{P:g}, and \eqref{P:i} are not. This means that $\hat\delta_{ji}$ is fractional and thus, by earlier argument, at most two of $(\ref{P:a}-\ref{P:d})$ can be tight. This leaves at most three tight constraints and this case has no extreme point.

\item[\underline{II.ii.b}] In this case, w.l.o.g, assume that constraints \eqref{P:e} and \eqref{P:g} are tight while \eqref{P:f} and \eqref{P:i} are not. First, suppose that \eqref{P:a} and \eqref{P:c} are tight alongside \eqref{P:g}:
	\begin{gather*}
	\begin{cases}
	c_i  \EQ  \LB_i(1-\hat\delta_{ji}) + (\LB_j + \PM_{ji})\hat\delta_{ji}
	\\
	c_j  \EQ  \LB_j(1-\hat\delta_{ij}) + (\LB_i + \PM_{ij})\hat\delta_{ij}
	\\
	\hat\delta_{ij} + \hat\delta_{ji}  \EQ  1
	\end{cases}
	\\
	\begin{cases}
	c_i  \EQ  \LB_i(1-\hat\delta_{ji}) + (\LB_j + \PM_{ji})\hat\delta_{ji}
	\\
	c_j  \EQ  \LB_j\hat\delta_{ji} + (\LB_i + \PM_{ij})(1-\hat\delta_{ji})
	\end{cases}
	\\
	\begin{cases}
	c_i  \EQ  \LB_i + (\LB_j + \PM_{ji} - \LB_i)\hat\delta_{ji}
	\\
	c_j  \EQ  \LB_i + \PM_{ij} + (\LB_i + \PM_{ij} + \LB_j)\hat\delta_{ji}
	\end{cases}
	\\
	c_i - c_j  \EQ  \LB_i + (\LB_j + \PM_{ji} - \LB_i)\hat\delta_{ji} - \LB_i - \PM_{ij} - (\LB_i + \PM_{ij} + \LB_j)\hat\delta_{ji}
	\\
	c_i - c_j  \EQ  (\LB_j + \PM_{ji} - \LB_i - \LB_i - \PM_{ij} - \LB_j)\hat\delta_{ji} - \PM_{ij} 
		\\
	c_i + \PM_{ij} - c_j  \EQ  (\PM_{ji} - \LB_i - \LB_i - \PM_{ij})\hat\delta_{ji}.
	\end{gather*}
Subtracting this from the tight form of \eqref{P:e} gives a contradiction.
	\begin{align*}
	\BM_{ij}&(1-\delta_{ij})  \EQ  (\PM_{ji} - \LB_i - \LB_i - \PM_{ij})\hat\delta_{ji}
	\\
	\UB_i &+ \PM_{ij} - \LB_{j}  \EQ  \PM_{ji} - 2\LB_i - \LB_i - \PM_{ij}
	\\
	0  &\EQ  \PM_{ji} - (\UB_i - \LB_{j}) - 2\LB_i - 2\PM_{ij}
	\\
		&\LS  0 - 2\LB_i - 2\PM_{ij}  \qquad\text{since }\PM_{ji} < (\UB_i - \LB_{j})\\
		&\LS  0
	\end{align*}
A similar contradiction can be achieved by supposing that \eqref{P:b} and \eqref{P:d} are tight instead. Thus, by combining with earlier arguments, at most one of $(\ref{P:a}-\ref{P:d})$ can be tight. This leaves at most three tight constraints and this case has no extreme point.
\end{itemize}

\noindent Having completed the case analysis of the projection \eqref{P}, it remains only to rigorously establish that the full dimensional formulation \ref{model:SU} is ideal. To do this, consider a fractional solution $(\mathbf{\hat c},\boldsymbol{\hat\delta})$ to the continuous relaxation of \ref{model:SU} and assume, for the sake of contradiction, that it is an extreme point. That is, assume that $(\mathbf{\hat{c}},\boldsymbol{\hat\delta})$ is tight to at least seven linearly independent inequalities in addition to satisfying equation \eqref{SU:disj}. There must exist a dimension $\hat s \in \{x,y\}$ in which at least four of these tight inequalities exist; call the other dimension $\hat t$. Each of these dimensions projects into its own variant of \eqref{P} and there are then four cases to consider:
\begin{enumerate}
\item A fractional $\hat\delta_{\hat s}$ necessitates a fractional extreme point in the projected system \eqref{P} which is a contradiction as we have already demonstrated that \eqref{P} has no fractional extreme points.

\item If \eqref{P:g} is not tight in direction $\hat t$, then a fractional $\hat\delta_{\hat t}$ implies a fractional $\hat\delta_{\hat s}$ by equality constraint \eqref{SU:disj}. We have already shown that this is a contradiction.

\item On the other hand, if \eqref{P:g} is tight in the direction $\hat t$, then a fractional $\hat\delta_{\hat t}$ implies that neither \eqref{P:h} nor \eqref{P:i} is tight in the direction $\hat t$ while both are tight in the direction of $\hat s$. However, \eqref{SU:disj} is linearly dependent with a tight \eqref{P:g} in the direction $\hat t$ and tight \eqref{P:h} and \eqref{P:i} in the direction of $\hat s$. 

Since \eqref{P} has only four variables it cannot support more than four linearly independent tight constraints, so such a fractional solution is only an extreme point if it is tight to three constraints in the direction $\hat t$ in addition to \eqref{P:g}. This is not achievable: by the proof of case \underline{II.i.a} above, we know that only one of \eqref{P:e} and \eqref{P:f} may be tight; further, by the proof of case \underline{II.ii.ab}, we also know that at most one of \eqref{P:a}-\eqref{P:d} may be tight.

\item Finally, if $\hat t$ has fewer than three of the tight inequalities; then $\hat s$ must have at least five. But the projected system \eqref{P} has only four variables and cannot support five or more linearly independent, tight constraints.
\end{enumerate}
\end{proof}

\section{Strong Forms of Lemmas \ref*{lem:SUCoversNecessary} and \ref*{lem:RUCoversNecessary}}\label{app:StongLemmas}
Here we present alternate forms of the lemmas which describe linearly dependent collections of constraints for some of the formulations. In particular, these variants (and their proofs) identify the conditions under which these collections are minimal. This additional information is not necessary to the body of the paper but is nonetheless interesting.

\vspace{1em}\noindent\textbf{Strong Form of Lemma \ref{lem:SUCoversNecessary}.}\emph{
    The collection of constraints $\inst{rSU:lb}{}$, $\inst{rSU:ub}{}$, $\inst{rSU:prec}{}$, and $\inst{rSU:indic}{}$ is linearly dependent for any given index $(k,l,s)\in\combs_{ij}$. Further, if $\PM_{kls} \neq \UB_{ls} - \LB_{ks}$ then the set is minimally dependent. 
}
\begin{proof}
    Assume that $\inst{rSU:lb}{}$, $\inst{rSU:ub}{}$, and $\inst{rSU:prec}{}$ are each tight and consider the combination $\inst{rSU:prec}{} - \inst{rSU:ub}{} + \inst{rSU:lb}{}$:
		\begin{align*}
		\big[c_{ks} - c_{ls}  &\EQ  \UB_{ks} - \LB_{ls} + (\LB_{ls} - \PM_{kls} - \UB_{ks})\delta_{kls}\big] \\
		- \big[c_{ks}  &\EQ  \UB_{ks} + (\UB_{ls} - \PM_{kls} - \UB_{ks})\delta_{kls}\big] 
		+ \big[c_{ls}  \EQ  \LB_{ls} + (\LB_{ks} + \PM_{kls} - \LB_{ls})\delta_{kls}\big] \\
		\Rightarrow\quad&\big[0  \EQ  (\LB_{ls} - \PM_{kls} - \UB_{ks} - \UB_{ls} + \PM_{kls} + \UB_{ks} + \LB_{ks} + \PM_{kls} - \LB_{ls})\delta_{kls}\big] \\
		\Rightarrow\quad&\big[0  \EQ  (\PM_{kls} - \UB_{ls} + \LB_{ks})\delta_{kls}\big]
		\end{align*}
    which is satisfied only if we have $\PM_{kls} = \UB_{ls} - \LB_{ks}$ or $\inst{rSU:indic}{}$ is also assumed to be tight.
\end{proof}

\vspace{1em}\noindent\textbf{Strong Form of \Cref{lem:RUCoversNecessary}.}\emph{
    The following collections of constraints are linearly dependent for all indices $(k,l,s) \in J$:}
    \begin{enumerate}[label=(\roman*),font=\itshape]
    \item $\inst{RUR:disj}{}$, $\inst{RUR:tight}{s}$, $\inst{RUR:indic}{kls'}$, and $\inst{RUR:indic}{lks'}$ where $s'\in\{x,y\}\setminus s$;
    
    \item $\inst{RUR:lb}{kls}$, $\inst{RUR:lb}{lks}$, $\inst{RUR:prec}{kls}$, $\inst{RUR:tight}s$, and $\inst{RUR:indic}{lks}$;
    
    \item $\inst{RUR:lb}{kls}$, $\inst{RUR:ub}{kls}$, $\inst{RUR:prec}{kls}$, $\inst{RUR:tight}{s}$, and $\inst{RUR:indic}{kls}$; and
            
    \item $\inst{RUR:ub}{kls}$, $\inst{RUR:ub}{lks}$, $,\inst{RUR:prec}{kls}$, $\inst{RUR:tight}{s}$, and $\inst{RUR:indic}{lks}$.
    \end{enumerate}
    \emph{Further, (i) is minimal, and (ii)-(iv) are minimal if $\PM_{lks} \neq \UB_{ks} - \LB_{ls}$ for all $(k,l,s)\in\combs_{ij}$.
}
\begin{proof} We demonstrate each collection by assuming tightness and considering a linear combination of the constraints:
    \begin{enumerate}[label=\emph{(\roman*)}]
    \item Assuming tightness, the combination $\inst{RUR:disj}{} - \inst{RUR:tight}{s} -  \inst{RUR:indic}{kls'} - \inst{RUR:indic}{lks'}$ reduces to $0 = 0$. However, removing any one of the constraints will leave us with only one instance of at least one variable, so this collection is minimal.
    
    \item Again, assume tightness and consider the combination $\inst{RUR:prec}{kls}) - \inst{RUR:lb}{lks} + \inst{RUR:lb}{kls}$:
        \begin{align*}
        [c_{kx} &- c_{lx}  \EQ  \PM_{lks} - (\PM_{kls} + \PM_{lks})\delta_{kls} + (\UB_{ks} - \PM_{lks} - \LB_{ls})\delta_{lks}] \\
            &\quad- [c_{ks}  \EQ  \LB_{ks} + (\LB_{ls} + \PM_{lks} - \LB_{ks})\delta_{lks}]\\
            &\qquad+ [c_{ls}  \EQ  \LB_{ls} + (\LB_{ks} + \PM_{kls} - \LB_{ls})\delta_{kls}] \\
            &\Rightarrow [0  \EQ  \LB_{ls} + \PM_{lks} - \LB_{ks} - (\LB_{ls} + \PM_{lks} - \LB_{ks})\delta_{kls} \\
            &\qquad + (\UB_{ks} - 2\PM_{lks} - 2\LB_{ls} + \LB_{ks})\delta_{lks}].
        \end{align*}
    But if $\inst{RUR:tight}{s}$ is also tight, then we have $\delta_{kls} = 1-\delta_{lks}$, so
        \begin{align*}
        [0  &\EQ  \LB_{ls} + \PM_{lks} - \LB_{ks} - (\LB_{ls} + \PM_{lks} - \LB_{ks})(1-\delta_{lks}) \\
            &\qquad + (\UB_{ks} - 2\PM_{lks} - 2\LB_{ls} + \LB_{ks})\delta_{lks}]\\
            &\quad\Rightarrow [0  \EQ  (\UB_{ks} - \PM_{lks} - \LB_{ls})\delta_{lks}]]
        \end{align*}
    which is satisfied only if we have $\PM_{lks} = \UB_{ks} - \LB_{ls}$ or $\inst{RUR:indic}{lks}$ is also assumed to be tight.
    
    \item Similarly, assume tightness and consider the combination $\inst{RUR:prec}{kls}) - \inst{RUR:ub}{kls} + \inst{RUR:lb}{kls}$:
        \begin{align*}
        [c_{kx} &- c_{lx}  \EQ  \PM_{lks} - (\PM_{kls} + \PM_{lks})\delta_{kls} + (\UB_{ks} - \PM_{lks} - \LB_{ls})\delta_{lks}] \\
            &\quad- [c_{ks}  \EQ  \UB_{ks} + (\UB_{ls} - \PM_{kls} - \UB_{ks})\delta_{kls}]\\
            &\qquad+ [c_{ls}  \EQ  \LB_{ls} + (\LB_{ks} + \PM_{kls} - \LB_{ls})\delta_{kls}] \\
            &\Rightarrow [0  \EQ  \LB_{ls} + \PM_{lks} - \UB_{ks} + (\LB_{ks} + \PM_{kls} - \UB_{ls} + \UB_{ks} - \PM_{lks} - \LB_{ls})\delta_{kls}\\
            &\qquad+ (\UB_{ks} - \PM_{lks} - \LB_{ls})\delta_{lks}].
        \end{align*}
    But if $\inst{RUR:tight}{s}$ is also tight, then we have $\delta_{lks} = 1-\delta_{kls}$, so
        \begin{align*}
        [0  &\EQ  \LB_{ls} + \PM_{lks} - \UB_{ks} + (\LB_{ks} + \PM_{kls} - \UB_{ls} + \UB_{ks} - \PM_{lks} - \LB_{ls})\delta_{kls}\\
            &\qquad+ (\UB_{ks} - \PM_{lks} - \LB_{ls})(1-\delta_{kls})] \\
            &\quad\Rightarrow[0  \EQ  (\LB_{ks} + \PM_{kls} - \UB_{ls})\delta_{kls}]
        \end{align*}
    which is satisfied only if we have $\PM_{kls} = \UB_{ls} - \LB_{ks}$ or $\inst{RUR:indic}{kls}$ is also assumed to be tight.
    
    \item Finally, assume tightness and consider the combination $\inst{RUR:prec}{kls} - \inst{RUR:ub}{kls} + \inst{RUR:ub}{lks}$:
        \begin{align*}
        [c_{kx} &- c_{lx}  \EQ  \PM_{lks} - (\PM_{kls} + \PM_{lks})\delta_{kls} + (\UB_{ks} - \PM_{lks} - \LB_{ls})\delta_{lks}] \\
            &\quad- [c_{ks}  \EQ  \UB_{ks} + (\UB_{ls} - \PM_{kls} - \UB_{ks})\delta_{kls}]\\
            &\qquad+ [c_{ls}  \EQ  \UB_{ls} + (\UB_{ks} - \PM_{lks} - \UB_{ls})\delta_{lks}] \\
            &\Rightarrow [0  \EQ  \UB_{ls} + \PM_{lks} - \UB_{ks} + (\UB_{ks} - \PM_{lks} - \UB_{ls})\delta_{kls}\\
            &\qquad+ (2\UB_{ks} - 2\PM_{lks} - \LB_{ls} - \UB_{ls})\delta_{lks}].
        \end{align*}
    But if $\inst{RUR:tight}{s}$ is also tight, then we have $\delta_{kls} = 1-\delta_{lks}$, so
        \begin{align*}
        [0  &\EQ  \UB_{ls} + \PM_{lks} - \UB_{ks} + (\UB_{ks} - \PM_{lks} - \UB_{ls})(1-\delta_{lks})\\
            &\qquad+ (2\UB_{ks} - 2\PM_{lks} - \LB_{ls} - \UB_{ls})\delta_{lks}] \\
            &\quad\Rightarrow [0  \EQ  (\UB_{ks} - \PM_{lks} - \LB_{ls})\delta_{lks}]
        \end{align*}
    which is satisfied only if we have $\PM_{lks} = \UB_{ks} - \LB_{ls}$ or $\inst{RUR:indic}{lks}$ is also assumed to be tight.
    \end{enumerate}
\end{proof}

\end{appendices}
\end{document}